\documentclass[11pt]{article}
\usepackage{arxiv}

\usepackage[utf8]{inputenc} 
\usepackage[T1]{fontenc}    

\usepackage{ bbold }
\usepackage{graphics}

\usepackage{pst-node}
\usepackage{tikz-cd}
\usepackage{enumitem}
\usepackage{wrapfig,color}

\usepackage{amsmath}
\usepackage{amssymb}
\usepackage{amsthm}
\usepackage{xspace}
\usepackage[normalem]{ulem}
\usepackage{graphicx}
\usepackage{epsfig}
\usepackage{ifpdf}
\usepackage{soul,hyperref}
\usepackage{latexsym}
\usepackage[mathscr]{euscript}
\usepackage{xspace}
\usepackage{color}
\usepackage{makeidx}
\usepackage{wrapfig,color}
\usepackage{stackrel}
\usepackage[linesnumbered,algoruled,boxed,lined]{algorithm2e}
\usepackage[noend]{algpseudocode}
 \usepackage[toc,page]{appendix}

\usepackage[all]{xy}

\usepackage{mathtools}

\usepackage{extarrows}
\usepackage{kbordermatrix}
\usepackage{todonotes}

\usepackage{quiver}

\title{Computing Generalized Ranks of Persistence Modules via Unfolding to Zigzag Modules}

\author{
    Tamal K. Dey
     \thanks{This research is supported by NSF grants CCF 2049010, DMS 2301360, and CCF 2437030}\\
\texttt{tamaldey@purdue.edu} 
 \And
 Cheng Xin\\
\texttt{xinc@purdue.edu} 
}
\date{Department of Computer Science\\ Purdue University} 

\theoremstyle{plain}
\newtheorem{theorem}{Theorem}[section]
\newtheorem{proposition}{Proposition}[section]
\newtheorem{observation}{Observation}[section]
\newtheorem{definition}{Definition}[section]
\newtheorem{fact}{Fact}[section]
\newtheorem{remark}{Remark}[section]
\newtheorem{notation}{Notation}[section]

 \newcommand{\propositionofref}{}

 \newcommand{\theoremofref}{}

\newcommand{\Int}{\mathbb{I}}
\newcommand{\Jnt}{\mathbb{J}}

\newcommand{\N}{\mathbb{N}}

\newcommand{\vecsp}{\mathbf{vec}_{\mathbb{F}}}
\newcommand*{\rom}[1]{\expandafter\@slowromancap\romannumeral #1@}

\newcommand{\field}[1] {{\mathbb{#1}}}

\newcommand{\rank}{\mathrm{rank}}

\newcommand{\prt}{\mathrm{prt}}

\newcommand{\M}{\mathbb{M}}
\newcommand{\Q}{\mathbb{Q}}
\newcommand{\U}{\mathbb{U}}
\newcommand{\D}{\mathcal{D}}

\newcommand{\Fld}{\mathrm{Fld}}

\newcommand{\Z}{\mathbb{Z}}
\newcommand{\bs}{{\sf b}}

\newcommand{\rk}{{\sf{rk}}}

\newcommand{\Pb}{P}
\newcommand{\Cb}{{\mathcal C}}

\makeatletter
\newcommand{\colim@}[2]{%
  \vtop{\m@th\ialign{##\cr
    \hfil$#1\operator@font colim$\hfil\cr
    \noalign{\nointerlineskip\kern1.5\ex@}#2\cr
    \noalign{\nointerlineskip\kern-\ex@}\cr}}%
}

\newcommand{\limproj}{\mathsf{lim}}
\newcommand{\colim}{\mathsf{colim}}

\definecolor{darkred}{rgb}{1, 0.1, 0.3}
\definecolor{darkblue}{rgb}{0.1, 0.1, 1}

\newcommand{\cancel}[1]

\begin{document}

\maketitle
\begin{abstract}
For a persistence module $\M$ indexed by a finite poset $P$, its generalized rank is defined as the rank of the limit-to-colimit map for the diagram of constituent vector spaces and linear maps of $\M$
over the poset $P$.  
For $2$-parameter persistence
modules, a zigzag persistence based algorithm has been
proposed that takes advantage of the fact that generalized rank for $2$-parameter modules is equal to the number of \emph{full} intervals in a zigzag module defined on the `boundary' of the poset. Analogous definition of boundary for $d$-parameter persistence modules or general poset indexed persistence modules does not seem plausible. To overcome this difficulty, we first \emph{unfold} a given module $\M$ into a zigzag module $\M_{ZZ}$ and then check how many full interval modules in a decomposition of $\M_{ZZ}$ can be \emph{folded back} to remain \emph{full} in a decomposition of $\M$. This number determines the generalized rank of $\M$. We give an efficient
algorithm to compute this number 
for the case when $\M$ is induced by a filtration of a simplicial complex over
a finite poset.
For special cases where $\M$ is the degree-$d$ homology functor on a filtration of a
$d$-complex, we obtain a more efficient algorithm including a linear time
algorithm for degree-$1$ homology for graphs.

\end{abstract}
\section{Introduction}
It is well known that one-parameter persistence modules decompose into
interval modules whose supports constitute \emph{persistence diagram}, or equivalently
\emph{barcode} of the given module, a fundamental object of study in topological data analysis~\cite{BS14,DW22,EH2010,Oudot15}. There are many situations in which persistence modules are parameterized over more than one parameter~\cite{BL23,carlsson2009theory,escolar2016persistence,lesnick2015theory,miller2019modules}.
Unfortunately, such \emph{multiparameter
persistence modules} do not necessarily admit a nice decomposition into intervals only. Instead,
they may decompose into indecomposables that are more complicated~\cite{carlsson2009theory}.
To overcome this difficulty, inspired by the work of Patel~\cite{patel2018generalized},
Kim and M\'{e}moli~\cite{kim2018generalized} proposed a decomposition
of poset-indexed modules (satisfying some mild condition)
into \emph{signed} interval modules.
In analogy to the one-parameter case, the supports
of these signed interval modules with multiplicity
are called the \emph{signed barcode} of the given module~\cite{botnan2021signed}.
The multiplicity of the
intervals is given by the M\"{o}bius inversion of a rank-invariant function; see
\cite{kim2018generalized,patel2018generalized}.
Botnan, Oppermann, and Oudot~\cite{botnan2021signed} recently showed
that a unique minimal signed barcode of a given persistence module
in terms of rectangles can be computed efficiently and raised the question of efficient
computation of other types of signed barcodes or invariants; see the recent work
in~\cite{CKM24} for more exposition along this direction. 

At the core of computing these signed barcodes for a persistence module $\M$ is the
problem of computing the generalized rank for an interval $I$ that is defined as
the rank of the limit-to-colimit map for the diagram of vector spaces of the restricted module $\M|_I$; see~\cite{kim2018generalized}. Recently,
Dey, Kim, M\'{e}moli~\cite{DKM22} showed that, for a $2$-parameter persistence module,
this generalized rank is given by the number of \emph{full} intervals in the
decomposition of a zigzag module $\M_{ZZ}$ that is a restriction of $\M$.
This becomes immediately useful because there are efficient algorithms
known for computing the barcode of a zigzag module
~\cite{DH22,DHM25, maria2014zigzag,milosavljevic2011zigzag}. 
However,
this result is limited to $2$-parameter persistence modules because
the zigzag module needs to be defined on the
boundary of a ``two dimensional'' interval. Beyond the $2$-parameter, this boundary does not remain
to be a path and hence poses a challenge in defining an appropriate zigzag module.

In this paper, we address the above problem and present
an algorithm to compute generalized
rank efficiently for finite-dimensional modules indexed by finite posets. The approach
uses the idea of straightening up
the input persistence module into a module defined over a zigzag path.
We call the process {\em unfolding} the module. We compute a decomposition 
of the resulting zigzag module into interval modules using a known algorithm.
Then, we design an algorithm that aims to
\emph{fold} the full interval modules (supported on entire zigzag path) in the
decomposition of the zigzag module back to the original module.
The ones which fold successfully to a full interval summand (supported on the entire poset) in the
original module gives us the generalized rank 
according to a result of Chambers and Letscher~\cite{chambers2018persistent}.

A viable approach to compute the generalized rank would be to compute the
limit and colimit of a given persistence module separately, say with
the recent algorithm in~\cite{SHN23}, and then
compute the rank of the limit-to-colimit map. Finding an efficient
implementation of this approach remains open. Each of the 
computations for limit, colimit, and then the rank of the map between them
may incur considerable computational cost
in practice. Our approach also has three distinct computations, unfolding the
module, then computing a zigzag persistence followed by a folding process.
Among these, the first one is done by a simple graph traversal, the second one
can be done with recent efficient practical zigzag persistence algorithms~\cite{DH22,DHM25}. The folding process is the
only costly step for which we provide an efficient algorithm.

One could also argue that a full decomposition algorithm such as the one 
in~\cite{DX22}
or the well-known Meataxe algorithm can be used to compute the
generalized rank because they compute all full interval summands. However,
the algorithm in~\cite{DX22} does not work for nondistinctly graded modules and the
Meataxe algorithm has a high time complexity ($O(t^{18})$, $t$ is the maximum of poset and filtration sizes), as pointed out in~\cite{DX22}). Furthermore, these algorithms expect the input in matrix forms (presentations or linear maps) instead of filtrations that are common in practice. We alleviate these issues; more precisely, highlights of our approach are:
\begin{itemize}
    \item It introduces an unfolding/folding technique for $P$-modules, that is,
    finite dimensional persistence modules defined on a finite poset $P$, which may be of independent
    interest. 
    \item It provides an algorithm that, given a
    simplicial $P$-filtration of a simplicial complex inducing a $P$-module $\M$ by homology functor,
    computes the generalized rank $\rk(\M)$ in $O(t^{\omega+2})$ time,
    $\omega < 2.373$, where
    the description size of $P$ and the given filtration is at most $t$ (see Eq.~\eqref{eq:tdef} and Theorem~\ref{thm:main}). The algorithm
    does not need to go through an extra step of computing a presentation
    of $\M$ from its inducing $P$-filtration. In fact, currently all published efficient algorithms for computing presentations work for $2$-parameter modules ($P\subseteq \mathbb{R}^2,\Z^2$)~\cite{KR21,LW22} and not for a general $P$-module indexed by a finite poset $P$.
    \item It computes
full interval summands of $\M$ representing its ``global sections" supported on $P$.
\item It gives a more efficient $O(t^\omega)$ algorithm for
the special cases of degree $d$-homology for $d$-complexes and a linear time
algorithm for degree-$1$ homology in graphs.
\end{itemize}


\section{Preliminaries and idea} \label{sec:genrank}
\subsection{Persistence modules}
We consider finite dimensional persistence modules indexed by connected finite posets. 
\begin{definition}For a poset $P$, let $\leq_P$ denote the partial order defining it. We also treat $P$ as a category with every $p\in P$ as its object and $\leq_P$ inducing the morphisms between them. Two points $p\leq_P q$ in $P$ are called \emph{immediate} and written $p\rightarrow q$ or $q \leftarrow p$ if and only if there is no $r\in P$, $r\not\in\{ p,q\}$, satisfying $p\leq_P r \leq_P q$. We also
write $p\leftrightarrow q$ to denote that either $p\rightarrow q$ or
$p\leftarrow q$.
\end{definition}
\begin{definition}\label{def:intervals}
An \emph{interval} $I$ of a poset $P$ is a non-empty subset $I\subseteq P$ so that	
(i) $I$ is convex with the partial order of $P$, that is, if $p,q\in I$ and $p\leq_P r\leq_P q$, then $r\in I$ \label{item:convexity};
(ii) $I$ is \emph{connected}, that is, for any $p,q\in I$, there is a sequence $p=p_0,	p_2,\cdots,p_m=q$ of elements of $I$ with $p_i\leftrightarrow p_{i+1}$ for $i\in \{0,\cdots, m-1\}$\label{item:interval3}.
Assuming $P$ is finite and connected, $P$ is also an interval
called the \emph{the full interval}.
\end{definition}

\begin{definition}
Given a poset $P$, we define a $P$-module to be a functor $\M:P\rightarrow \vecsp$ where
$\vecsp$ is the category of finite dimensional vector spaces over a fixed field $\mathbb{F}$ with the
morphisms being the
linear maps among them. 
For two points
$p\leq_P q$, we also write $\M(p\leq_P q)$ to denote the 
morphisms of $\M$.
\end{definition}

\begin{definition}
A $P$-module $\M$ is called indecomposable if there is no direct sum $\M\cong\M_1\oplus \M_2$
so that both $\M_1$ and $\M_2$ are nontrivial $P$-modules.
\end{definition}
Any (pointwise) finite dimensional $P$-module is a direct sum of indecomposables with local
endomorphism ring~\cite{CraBoe94}; see also~\cite{BC20}. Such a decomposition is essentially unique 
up to automorphism according to
Azumaya-Krull-Remak-Schmidt theorem~\cite{azumaya1950corrections}:
(Also see \cite[Theorem 1.11]{Kirilov}).
\begin{theorem}
Every $P$-module has a unique decomposition up to isomorphism
$\M\cong \M_1\oplus\cdots\oplus \M_k$ where each $\M_i$ is an indecomposable module. 
\label{thm:krull-schmidt}
\end{theorem}
\fbox{\parbox{\textwidth}{
For a decomposable module $\M$, there exist submodules $\M_i$, $i\in [k]$, of $\M$ with inclusions
$j_i: \M_i\rightarrow \M$ so that $(j_i: \M_i\rightarrow \M)_{i\in [k]}$ make $\M$ the direct sum of
$\M_i$s written as $\M=\M_1\oplus\cdots\oplus \M_k$. In light of this,
we replace the isomorphism in Theorem~\ref{thm:krull-schmidt} $\M\cong\M_1\oplus\cdots\oplus \M_k$
with equality $\M=\M_1\oplus\cdots\oplus \M_k$ where each $\M_i$ is indecomposable and
is a \emph{submodule} of $\M$; see Figure~\ref{fig:genrank}.
We call such a decomposition an \emph{internal direct decomposition} or
simply \emph{direct decomposition} denoted $\D:=\D(\M)$. We treat
$\D$ as the set $\{\M_i\}_{i\in[k]}$ of indecomposables if there is no confusion.
Notice that the uniqueness of such a decomposition
is up to automorphisms of $\M$ (and permutations of $\M_i$s). This
aspect plays an important role in our algorithm to follow.}
}

\vspace{0.1in}
For an interval $I$ of $P$, any module $\Int^I:P\rightarrow \vecsp$
is an \emph{interval module} if:
\begin{equation*}
\Int^I(p)\cong\begin{cases}
\mathbb{F}&\mbox{if}\ p\in I,\\0
&\mbox{otherwise,} 
\end{cases}\hspace{15mm} \Int^I(p\leq_P q)\cong\begin{cases} \mathrm{id}_\mathbb{F}& \mbox{if} \,\,p,q\in I,\ p\leq_P q,\\ 0&\mbox{otherwise.}\end{cases}
\end{equation*}

\begin{definition}An interval module $\Int=\Int^P$ $(\Int(p)\neq 0$ for any $p\in P)$ is a \emph{full interval module}(Figure~\ref{fig:genrank}).
\end{definition}
\begin{figure}[htbp]
\centerline{\includegraphics{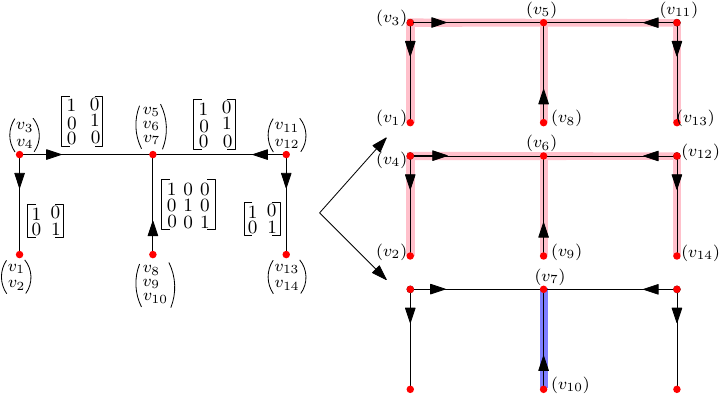}}
\caption{(left) A $P$-module with vector spaces generated by bases shown at the points of $P$ (arrows represent the partial order)
and internal maps shown as matrices; 
(right) a direct decomposition that contains
two full interval modules (top,middle) and another interval module (bottom) which is not full.}
\label{fig:genrank}
\end{figure}

\subsection{Limits and colimits and generalized rank}
\label{sec:limits-colimits}

We recall the notions of limit and colimit from category theory \cite{mac2013categories}.
Although it is known that limit and colimit
may not exist for all functors, they do exist for functors defined on finite posets, which
is the case we consider. The following definitions are reproduced
from~\cite{DKM22}. Let $\Cb$ denote a category in the following definitions.

\begin{definition}[Cone]\label{def:cone} Let $F:\Pb\rightarrow \Cb$ be a functor. A \emph{cone} over $F$ is a pair $\left(L,(\pi_p)_{p\in P}\right)$ consisting of an object $L$ in $\Cb$ and a collection $(\pi_p)_{p\in P}$ of morphisms $\pi_p:L \rightarrow F(p)$ that commute with the arrows in the diagram of $F$, i.e. if $p\leq q$ in $\Pb$, then $\pi_q= F(p\leq q)\circ \pi_p$ in $\Cb$, i.e. the diagram below commutes.
\end{definition}
\begin{equation}\label{eq:cone commutativity}
    \begin{tikzcd}F(p)\arrow{rr}{F(p\leq q)}&&F(q)\\
& L \arrow{lu}{\pi_p} \arrow{ru}[swap]{\pi_q}\end{tikzcd}
\end{equation}

The cone $\left(L,(\pi_p)_{p\in \Pb}\right)$ over $F$ sometimes is denoted simply by $L$, suppressing the morphisms $(\pi_p)_{p\in \Pb}$ if it causes no confusion. A limit of $F:\Pb\rightarrow \Cb$ is an object in the collection of all cones over $F$ that is terminal:

\begin{definition}[Limit]\label{def:limit} Let $F:\Pb\rightarrow \Cb$ be a functor. A \emph{limit} of $F$ is a cone over $F$, denoted by $\left(\limproj\, F,\ (\pi_p)_{p\in \Pb} \right)$ or simply $\limproj\, F$, with the following (universal) \emph{terminal} property: 
For any cone $\left(L',(\pi'_p)_{p\in \Pb} \right)$ of $F$, there is a \emph{unique} morphism $u:L'\rightarrow \limproj\, F$ such that $\pi_p'=\pi_p\circ u$ for all $p\in \Pb$.
\end{definition}

Cocones and colimits are defined in a dual manner:
\begin{definition}[Cocone]\label{def:cocone} Let $F:\Pb\rightarrow \Cb$ be a functor. A \emph{cocone} over $F$ is a pair $\left(C,(i_p)_{p\in \Pb}\right)$ consisting of an object $C$ in $\Cb$ and a collection $(i_p)_{p\in \Pb}$ of morphisms $i_p:F(p)\rightarrow C$ that commute with the arrows in the diagram of $F$, i.e. if $p\leq q$ in $\Pb$, then $i_p= i_q\circ F(p\leq q)$ in $\Cb$, i.e. the diagram below commutes.
\end{definition}
\begin{equation}\label{eq:cocone commutativity}
    \begin{tikzcd}
&C\\
F(p)\arrow{ru}{i_p}\arrow{rr}[swap]{F(p\leq q)}&&F(q)\arrow{lu}[swap]{i_q}
\end{tikzcd}
\end{equation}

In Definition \ref{def:cocone}, a cocone $\left(C,(i_p)_{p\in \Pb}\right)$ over $F$ is sometimes denoted simply by $C$, suppressing the collection $(i_p)_{p\in \Pb}$ of morphisms. A colimit of $F:\Pb\rightarrow \Cb$ is an initial object in the collection of cocones over $F$:

\begin{definition}[Colimit]\label{def:colimit}Let $F:\Pb\rightarrow \Cb$ be a functor. A \emph{colimit} of $F$ is a cocone, denoted by $\left(\colim\, F,\ (i_p)_{p\in \Pb}\right)$ or simply $\colim\, F$, with the following \emph{initial} property: If there is another cocone $\left(C', (i'_p)_{p\in \Pb}\right)$ of $F$, then there is a \emph{unique} morphism $u:\colim\, F\rightarrow C'$  such that $i'_p=u\circ i_p$ for all $p\in \Pb$.
\end{definition}

The following proposition gives a standard way of constructing a limit of a $P$-module $\M$. See for example \cite{DKM22,kim2018generalized}).

\begin{notation}\label{not:sim}
Let $p,q\in P$ and let $v_{p}\in \M(p)$ and $v_{q}\in \M(q)$. We write $v_p\sim v_q$ if $p$ and $q$ are comparable, and either $v_p$ is mapped to $v_q$ via $\M(p\leq_P q)$ or $v_q$ is mapped to $v_p$ via $\M(q\leq_P p)$.
\end{notation}

\begin{proposition}\label{prop:alternate-limit} 
	\begin{enumerate}[label=(\roman*)]
		\item \label{item:limit}The limit of $\M$ is (isomorphic to) the pair $\left(W,(\pi_p)_{p\in P}\right)$ described as follows: 
		\begin{equation}\label{eq:limit}
		    W:=\left\{(v_p)_{p\in P}\in \bigoplus_{p\in P} \M(p):\ \forall p\leq q \mbox{ in } P,\ v_{p}\sim v_{q} \right\}
		\end{equation}
		and for each $p\in P$, the map $\pi_p:W\rightarrow \M(p)$ is the canonical projection. 
		An element of $W$ is called a \emph{global section} of $\M$.
		\item The colimit of $\M$ is (isomorphic to) the pair $\left(U, (i_p)_{p\in \Pb}\right)$ described as follows: For $p\in \Pb$, let the map $j_p:\M(p)\hookrightarrow \bigoplus_{p\in \Pb}\M(p)$ be the canonical injection. The space $U$ is the quotient  space $\left(\bigoplus_{p\in \Pb}\M(p)\right)/T$, where $T$ is the subspace of $\bigoplus_{p\in \Pb} \M(p)$ which is generated by the vectors of the form  $j_p(v_p)-j_q(v_q), \ v_p\sim v_{q},$ the map $i_p:\M(p)\rightarrow U$ is the composition $\rho\circ j_p$, where $\rho$ is the quotient map $\bigoplus_{p\in \Pb} \M(p)\rightarrow U$.
		\label{item:colim}
	\end{enumerate}
\end{proposition} 

These definitions imply that, for every $p\leq_P q$ in $P$, $\M(p\leq_P q)\circ \pi_p =\pi_q\ \ \mbox{and }\  i_q\circ \M(p\leq_P q) =i_p$, which in turn imply $i_p \circ \pi_p=i_q \circ \pi_q:L\rightarrow C$ for any $p,q\in P$. 

\begin{definition}[{\cite{kim2018generalized}}]\label{def:rank}
The \emph{canonical limit-to-colimit map $\psi_\M:\limproj\, \M\rightarrow \colim\, \M$} is the linear map $i_p\circ \pi_p$ for \emph{any} $p\in P$. The \emph{generalized rank} of $\M$ is $\rk(\M):=\rank(\psi_\M)$. 
\end{definition}

The following result allows us to compute $\rk(\M)$ as the number of the full interval modules in a direct decomposition of $\M$.
\begin{theorem}[{\cite[Lemma 3.1]{chambers2018persistent}}]
The rank $\rk(\M)$ is equal to the number of full interval modules in a direct decomposition of $\M$. 
\label{thm:rk}
\end{theorem}

\subsection{Idea of unfolding to zigzag module}
The overall idea of our approach is to ``straighten up'' the given $P$-module $\M$ into
a zigzag module $\M_{ZZ}$ which is a zigzag module defined over
a linear poset $P_{ZZ}$. It is well known that a zigzag module like $\M_{ZZ}$ decomposes into interval modules and efficient algorithms for computing them exist. After
computing these interval modules, we attempt to fold back the full interval modules in this
decomposition of $\M_{ZZ}$
to full interval summands of the original module $\M$. 
We know that full interval modules that are summands of $\M$ unfolds into
full interval modules that are summands of $\M_{ZZ}$. However, the converse may not hold, that is, not
all full interval summands of $\M_{ZZ}$ fold back to full interval summands of $\M$. 
So, the main challenge
becomes to determine which full interval modules in the computed decomposition of $\M_{ZZ}$ do indeed
fold back (possibly with some modifications) to full interval summands of $\M$. 

\begin{figure}[htbp]
\centerline{\includegraphics{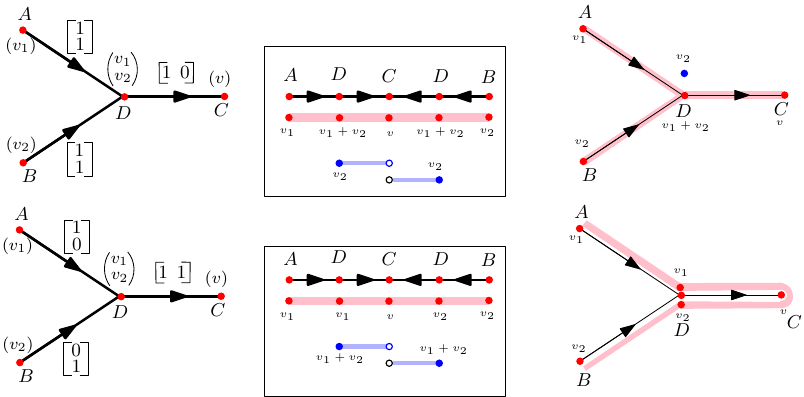}}
\caption{Unfolding two modules (defined over $\Z_2$) on the same poset into zigzag modules; hollow circles represent zero vectors. (top row) 
Decomposition of the
zigzag module folds back to two interval modules, one with support
on the entire poset giving rise to a generalized rank equal to one, and the other with support on $D$. (bottom row) Full interval module in the decomposition of the zigzag module does not
fold back to a full interval module (attempted folding shows the discrepancy at point $D$), giving rise to generalized rank equal to zero.}
\label{fig:unfoldingEx}
\end{figure}
Figure~\ref{fig:unfoldingEx} illustrates this idea. The module $\M$ is defined on
a four point poset $P=\{A,B,C,D\}$ shown on left. Bases of the vector spaces are shown in open brackets `$()$' and
linear maps in these bases are shown as matrices.
The poset is straightened into a zigzag path $A\rightarrow D \rightarrow C \leftarrow D \leftarrow B$. 
One way to look at this straightening
is to view $P$ as a directed graph and to traverse all its vertices and edges sequentially
possibly with repetition. Starting from the vertex $A$ and moving to an adjacent node 
disregarding the direction while noting down the visited node and the directed edge
produces the zigzag path. This process of unfolding a poset into a zigzag path is formalized
in section~\ref{sec:unfold}. The module $\M$ is unfolded into the zigzag module $\M_{ZZ}$ by copying the vector
spaces and linear maps at vertices and edges respectively into their unfolded versions. For the
module shown on the top, we get three interval modules (bars) in a decomposition of
$\M_{ZZ}$, the full interval module supported on $[A,B]$ and the other two supported on
two copies of
$D$ respectively. Bases of one dimensional vector spaces for the interval modules are indicated beneath
them. When we fold back $P_{ZZ}$ to $P$ (reversing the process of unfolding) sending $\M_{ZZ}$ to $\M$,
the full interval module does fold back to a summand that is a full interval module because the vectors $v_1+v_2$\footnote{this is a vector addition, not a direct sum} at two copies of $D$
are the same and hence map to the same vector in $\M(D)$. The other two 
interval modules supported on single-points in the decomposition of $\M_{ZZ}$ also
fold back to a summand of $\M$ supported on a single point $D$
where $\M(D)$ is generated by the vector $v_2$. 

The case for the module shown in the bottom row of Figure~\ref{fig:unfoldingEx} is not the same.
In this case, $\M_{ZZ}$ also has the barcode consisting of the same three intervals, but the corresponding interval
modules are not the same. The full interval module in this case has different vector spaces spanned by $v_1$ and $v_2$ respectively 
at the two copies of $D$. Thus, this interval module does not fold into a full interval submodule of $\M$ as an attempt on right indicates. We can determine such full interval modules
in a decomposition of the zigzag module by checking if the vectors at the copied vertices
are the same or not. However, even if this check succeeds, the interval module
may not fold to a summand of $\M$. Figure~\ref{fig:unfoldingEx3} (top and bottom) shows
such examples. On the other hand, 
even if this check fails, it may be possible to change
the full interval module to have the vectors to be the same
at the copied vertices. Figure~\ref{fig:unfoldingEx2} in section~\ref{sec:complete-limit} illustrates such an example. Determining
such cases and taking actions accordingly are key aspects of our algorithm.

\section{Folding and Unfolding}
\label{sec:folding-unfolding}
As already indicated, our algorithm operates on two 
main constructs called \emph{folding} and \emph{unfolding} of persistence
modules. In this section, we introduce these concepts formally for
general modules and then later specialize them to zigzag modules for the algorithm.

\begin{definition}
Let $Q$ be a finite poset. A poset $\Fld_s Q$ is a \emph{folded} poset of
$Q$ if there exists a surjection $s: Q\rightarrow \Fld_s Q$, which
(i) preserves order, that is, $p\leq_Q q$ only if $s(p)\leq_{\Fld_s Q} s(q)$ for all $p,q \in Q$, and (ii) surjects also on the Hasse diagram of $\Fld_s Q$, that is, for every immediate pair $u\rightarrow v$ in $\Fld_s Q$, there is a pair $p\leq_Q q$ where $s(p)=u$ and $s(q)=v$. We say $s$ is a folding of
$Q$ and $Q$ is an \emph{unfolded} poset of $\Fld_s Q$.
\end{definition}

Viewing posets as categories, a folding $s:Q\rightarrow \Fld_s Q$ can be viewed as a functor from $Q$ to $\Fld_s Q$.

\begin{definition}
Let $P=\Fld_s Q$ and $\M:P\to \vecsp$ be a $P$-module and $\N:Q\to \vecsp$ be a  $Q$-module. We say $\M$ is an $s$-folding of $\N$ $(\N$ $s$-folds or simply folds into $\M)$ and $\N$ is an $s$-unfolding of $\M$ $(\M$ $s$-unfolds or simply unfolds into $\N)$ 
if there is a folding $s:Q\rightarrow P$ so that
\begin{equation}
     \N(q)=\M(s(q))\, ~(equality~as~sets)~ \forall q\in Q\,;\, \N(p\leq_Q q)=\M(s(p)\leq_{P} s(q))\, 
    \forall (p \leq_Q q). \label{eq:folding}
\end{equation}
We write $\M=\Fld_s(\N)$ and $\N=\Fld_s^{-1}(\M)$.
\label{def:folding}
\end{definition}


\begin{remark}
Observe that for a given folding $s:Q\to P$, a $P$-module $\M$ always has an induced $s$-unfolding $\M\circ s$ by pre-composition with $s$. However, for a given $Q$-module $\N$, an $s$-folding may not exist because it may happen that $\N(q)\not =\N(q')$ where
$s(q')=s(q)$, or $\N(q_1\rightarrow q_2)\not = \N(q_1'\rightarrow q_2')$ where
$s(q_1)=s(q_1')$ and $s(q_2)=s(q_2')$.

An interesting and important fact is that two isomorphic modules 
may have different $s$-foldings. Figure~\ref{fig:unfoldingEx} shows such an example. 
Two zigzag modules shown in the middle are isomorphic (barcodes are the same), but they are not exactly the same as modules (even though vector spaces are pointwise equal, morphisms are not). 
So, even if
they are isomorphic, they fold to different modules as shown in left.
Nevertheless, if a folding exists, a module necessarily folds to a \emph{unique} module as
Proposition~\ref{prop:unique} states. 
\end{remark}

\cancel{
Following observation holds because of the uniqueness property of folding if it exists.
\begin{proposition}\label{prop:folding-summand}
Given a $Q$-module $\N=\N_1\oplus \N_2$ and a folding $s:Q\to P$, 
(i) if $\Fld_s(\N_1), \Fld_s(\N_2)$ exist, then $\Fld_s(\N)$
exists and $\Fld_s(\N)=\Fld_s(\N_1)\oplus \Fld_s(\N_2)$;
(ii) if $\Fld_s(\N_1)$ and $\Fld_s(\N)$ exist, then $\Fld_s(\N_2)$ exists
and $\Fld_s(\N)=\Fld_s(\N_1)\oplus \Fld_s(\N_2)$.
\end{proposition}
\begin{proof}
We prove (i) and (ii) can be proved similarly.
First, notice that, since $\Fld_s(\N_1)$ and
$\Fld_s(\N_2)$ exist, there are graded bases $B^{\N_1}$, $B^{\N_2}$, and internal linear
maps $B^{\N_1}\rightarrow B^{\N_1}$, $B^{\N_2}\rightarrow B^{\N_2}$ respectively 
which $s$-fold to vector spaces and internal linear maps of $\Fld_s(\N)$. Thus, $\Fld_s(\N)$ exists.
Since $\Fld_s(\N)\circ s=\N=\N_1\oplus \N_2=(\Fld_s(\N_1)\oplus \Fld_s(\N_2))\circ s$, by the uniqueness property of folding 
in Proposition~\ref{prop:unique_folding}, we know that $\Fld_s(\N)=\Fld_s(\N_1\oplus \N_2)=\Fld_s(\N_1)\oplus \Fld_s(\N_2)$.
\end{proof}
}
\begin{definition}
    Let $\M_1$ be a summand of $\M$. A complement of $\M_1$ denoted
    $\overline{\M_1}$ is a summand so that $\M=\M_1\oplus \overline{\M_1}$. Observe that $\overline{\M_1}$ is not necessarily unique though for a given decomposition, it is \emph{uniquely identified}.
\end{definition}
\begin{definition}
For a folding $s: Q \rightarrow \Fld_s(Q)$ and 
a $Q$-module $\N$, we say $\N$
is $s$-\emph{foldable} $($or simply foldable$)$ if $\N(q)=\N(q')$ for every pair $q,q'\in Q$ where $s(q)=s(q')$.
\end{definition}

We weaken the condition of foldability to \emph{invertibility} in the following definition.
\begin{definition}
For a folding $s: Q \rightarrow \Fld_s(Q)$ and a $Q$-module $\N$, a summand $\N'$
in a decomposition $\N=\N'\oplus \N''$ is called \emph{invertible} if there exists a foldable summand $\Q$ of $\N$
so that $\N=\Q\oplus \N''$.
\end{definition}

Foldability of a summand and its complement of a module guarantees that they both remain summands after the 
folding of the module. However, we can weaken the requirement on foldability for the complement.
If foldability holds for a summand 
then it remains a summand after the folding if and only if its complement
is invertible. These are the statements of the following theorem.

\begin{theorem}
Let $\M$ be a $P$-module and $\N$ be a $Q$-module where
$\M=\Fld_s(\N)$ for some folding $s:Q\rightarrow P$. 
\begin{enumerate}
\item If $\N=\N_1\oplus\overline{\N_1}$ and both $\N_1$ and $\overline{\N_1}$ are foldable, then $\Fld_s(\N_1)$ and $\Fld_s(\overline{\N_1})$ exist and $\M=\Fld_s(\N_1)\oplus\Fld_s(\overline{\N_1})$. 
\item Conversely, if $\M=\M_1\oplus \overline{\M_1}$, then $\Fld_s^{-1}(\M_1)$ and $\Fld_s^{-1}(\overline{\M_1})$ necessarily exist and 
$\N=\Fld_s^{-1}(\M_1)\oplus \Fld_s^{-1}(\overline{\M_1})$.
\item If $\N=\N_1\oplus \overline{\N_1}$ where $\N_1$ is foldable, then
$\Fld_s({\N_1})$ necessarily exists and is a summand of $\M$ if and only if
$\overline{\N_1}$ is invertible.

\end{enumerate}
\label{thm:full-folding}
\end{theorem}
In Figure~\ref{fig:unfoldingEx3} (bottom), the gray interval is foldable
and its complement (direct sum of the red and blue intervals) is foldable. So, the gray
interval and its complement fold into summands whereas in the top figure in Figure~\ref{fig:unfoldingEx3}, none of the interval modules folds into a summand
because even if the blue one is foldable, its complement red one is neither foldable nor
invertible.
\begin{figure}[htbp]
\centerline{\includegraphics{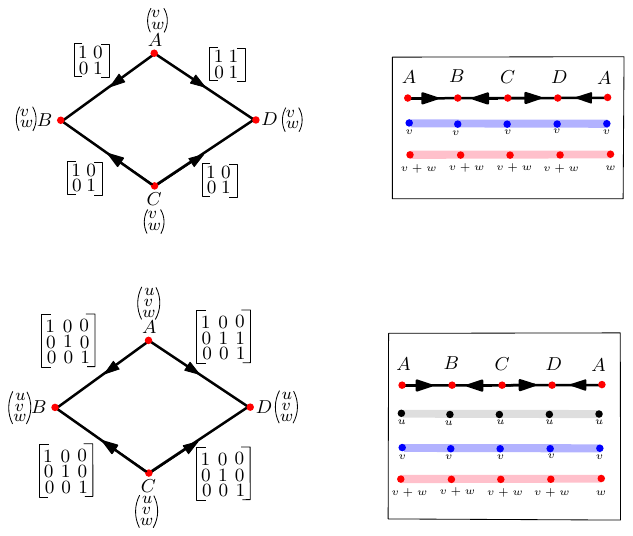}}
\caption{(left) $\M$, (right) $\M_{ZZ}$: (top) the blue interval module is foldable whereas its complement (red) is not invertible, so neither folds to a summand of $\M$; (bottom) the top gray interval module is foldable and its complement is foldable (hence invertible), so the gray interval module folds to a summand of $\M$. 
The blue interval module is foldable, but its complement is not
invertible (foldable), so it does not fold to a summand in $\M$.}
\label{fig:unfoldingEx3}
\end{figure}

The following propositions help proving Theorem~\ref{thm:full-folding}.

\begin{proposition}\label{prop:unique_folding}
For a $Q$-module $\N$ and a folding $s:Q\to P$, if $\Fld_s(\N)$ exists, then $\Fld_s(\N)$ is unique.
\label{prop:unique}
\end{proposition}
\begin{proof}
If there were two modules $\M_1$ and $\M_2$ that are $s$-foldings of $\N$, then
both $\M_1(p)=\M_2(p)=\N(s^{-1}(p))$ for every $p\in P$. Furthermore, $\M_1(p\leq_P p')=\M_2(p\leq_P p')=\N(s^{-1}(p)\leq_Q s^{-1}(p'))$ 
$\forall(p\leq_P p')$
by definition. This immediately
shows that $\M_1=\M_2$.
\end{proof}
 Proposition~\ref{prop:summand}, Proposition~\ref{prop:foldexist}, and Proposition~\ref{prop:submodule} below are used to prove Theorem~\ref{thm:full-folding} 
that characterizes modules
which fold into summand modules.
\begin{proposition}[\cite{carlsson2009zigzag}]
    Let $\N$ be a submodule of a $P$-module $\M$ where there is a submodule
    $\overline\N$ of $\M$ so that $\M(p)=\N(p)\oplus \overline{\N}(p)$ for every $p\in P$. Then, $\N$ is a summand of $\M$, that is, $\M=\N\oplus \overline{\N}$.
    \label{prop:summand}
\end{proposition}

\begin{proposition}
    Let $\N$ be a $Q$-module where $\Fld_s(\N)$ exists for some folding
    $s: Q\rightarrow P$. Then, for any submodule $\N'\subseteq \N$ that is foldable $\Fld_s(\N')$ exists. 
    \label{prop:foldexist}
\end{proposition}
\begin{proof}
    Construct a $P$-module $\M$ as follows: First, put $\M(p)=\N'(q)$ where $p=s(q)$. This is well defined because $\N'$ is foldable. Next, put
    $\M(p\leq_P p')=\N'(q\leq_Q q')$ where $p=s(q)$ and $p'=s(q')$. This is also well defined because for every pair $q\leq_Q q'$ so that $p=s(q)$ and
    $p'=s(q')$, we have
    $\N'(q\leq_Q q')$ to be a restriction of $\N(q\leq_Q q')$ where $\Fld_s(\N)$ exists. Observe that $\M=\Fld_s(\N')$ by Definition~\ref{def:folding}.
\end{proof}
\begin{proposition}
    Let $\M$ be a $P$-module and $\N$ be a $Q$-module where
    $\M=\Fld_s(\N)$ for some folding $s:Q\rightarrow P$. For any submodule $\N'\subseteq \N$, if $\Fld_s(\N')$ exists, then it is a submodule of $\M$. Conversely, if $\M'$ is a submodule of $\M$, then $\Fld_s^{-1}(\M')$ is a submodule of $\N$.
    \label{prop:submodule}
\end{proposition}
\begin{proof}
We have $\N'(q)\subseteq \N(q)$ for $\forall q\in Q$ and
$\N'(p \leq_Q q)$ is a restriction of $\N(p\leq_Q q)$ on $\N'(p)$ $\forall p\leq _Q q$ because $\N'$ is a submodule of $\N$. Then, by Definition~\ref{def:folding},
\begin{eqnarray*}
\Fld_s(\N')(s(q))&=&\N'(q)\subseteq\N(q)=\M(s(q))~~\forall q\in Q \mbox{ and }\\
\Fld_s(\N')(s(p)\leq_P s(q))&=&\N'(p\leq_Q q) =\N(p\leq_Q q)|_{\N'(p)}=\M(s(p)\leq_P s(q))|_{\Fld_s(\N')(s(p))},
\end{eqnarray*}
which establishes that $\Fld_s(\N')$ is a submodule of $\M$.

For the converse statement, check that $\Fld_s^{-1}(\M')$ necessarily exists and it is a submdoule of $\M$ by definition of unfolding.
\end{proof}
\begin{proof}[Proof of Theorem~\ref{thm:full-folding}]
~\\
   (1) By Proposition~\ref{prop:foldexist}, both $\Fld_s(\N_1)$ and $\Fld_s(\overline{\N_1})$ exist and
    they are submodules of $\M$ by Proposition~\ref{prop:submodule}. If we show that
    $\M(p)=\Fld_s(\N_1)(p)\oplus \Fld_s(\overline{\N_1})(p)$ for every $p\in P$, then $\M=\Fld_s(\N_1)\oplus \Fld_s(\overline{\N_1})$ due to Proposition~\ref{prop:summand}.

    We have $\M(p)=\Fld_s(\N)(p)=\Fld_s(\N_1\oplus \overline{\N_1})(p)=\Fld_s(\N_1)(p)\oplus \Fld_s(\overline{\N_1})(p)$. 

   (2) For the converse, observe that $\Fld_s^{-1}(\M_1)$ and $\Fld_s^{-1}(\overline{\M_1})$ necessarily exist and they are submodules of $\N$ by Proposition~\ref{prop:submodule}. Furthermore, $\N(q)=\Fld_s^{-1}(\M)(q)=(\Fld_s^{-1}(\M_1\oplus \overline{\M_1}))(q)= \Fld_s^{-1}(\M_1)(q)\oplus \Fld_s^{-1}(\overline{\M_1})(q)$. Then, by Proposition~\ref{prop:summand}, we have $\N=\Fld_s^{-1}(\M_1)\oplus \Fld_s^{-1}(\overline{\M_1})$. 

   (3) First, assume that $\Fld_s(\N_1)$ exists and is a summand of $\M$.
   Then, we have $\M=\Fld_s \N_1\oplus \overline{\Fld_s \N_1}$. Applying
   (2) above, we have $\N=\Fld_s^{-1}(\Fld_s\N_1)\oplus \Fld_s^{-1}(\overline{\Fld_s \N_1})
   = \N_1\oplus \Fld_s^{-1}(\overline{\Fld_s \N_1})$. Let $\Q=\Fld_s^{-1}(\overline{\Fld_s \N_1})$ which is foldable. Then,
   \begin{eqnarray}
   \N_1\oplus \overline{\N_1}=\N=\N_1\oplus \Q
   \label{eq:nq}
   \end{eqnarray}
   Since
   $\mathbb{Q}$ is foldable, we have that $\overline{\N_1}$ is invertible.

   Now, assume that $\overline{\N_1}$ is invertible. By definition, $\N=\N_1\oplus\Q$
   for some foldable summand $\Q$. Apply (1) above to
   conclude that $\Fld_s(\N_1)$ exists and is a summand of $\M$.

   \cancel{
   For contradiction,
   suppose that $\overline{\N_1}$ is not foldable and let $q$ and $q'$ be two points in $Q$ so that $s(q)=s(q')$ and 
   $\overline{\N_1}(q)\neq \overline{\N_1}(q')$. 
   
   Let $\{b_x^1,\ldots,b_x^k\}$ denote a basis for $\mathbb{U}(x)$ for any $x\in Q$.
   We can choose $\{b_q^1,\ldots, b_q^k\}$ and $\{b_{q'}^1,\ldots, b_{q'}^k\}$ so that
   $b_q^i=b_{q'}^i$ for every $i\in \{1,\ldots, k\}$ because $\mathbb{U}(q)=\mathbb{U}(q')$. 
   
   It follows from \eqref{eq:nq} that $\overline{\N_1}(x)\oplus \N_1(x)=\mathbb{U}(x)\oplus \N_1(x)$ for every $x\in Q$. Since $\N_1$ is a full interval module,
   $\dim \overline{\N_1}(x)=\dim \mathbb{U}(x)$.
    Then, we can choose a basis $\{c_q^1,\ldots, c_q^{k}\}$ for $\overline{\N_1}(q)$
    so that $c_q^i=\sum_{j=1}^k (\alpha_jb_q^j+\alpha'_j \bs_q^{\N_1})$ and
    a basis $\{c_{q'}^1,\ldots, c_{q'}^{k}\}$ for $\overline{N_1}(q')$ so that
    $c_{q'}^i=\sum_{j=1}^k (\alpha_jb_{q'}^j+\alpha'_j \bs_{q'}^{\N_1})$ for
    $\alpha_j,\alpha'_j\in \field{F}$. This means that for every
    $i\in \{1,\ldots, k\}$, $c_q^i=c_{q'}^i$ because $b_q^i=b_{q'}^i$ and
    $\bs^{\N_1}_q=\bs^{\N_1}_{q'}$ as $\N_1$ is foldable and full.
    It follows that 
    $\mathrm{span}\{c_q^1,\ldots, c_q^k\}=\mathrm{span}\{c_{q'}^1,\ldots,c_{q'}^k\}$.
    So, we have $\overline{\N_1}(q)=\overline{\N_1}(q')$ reaching a contradiction.

   To finish the proof, observe that both $\N_1$ and
   $\overline{\N_1}$ are foldable and thus $\Fld_s(\overline{\N_1})$ is a summand
   in $\M$.
   }
\end{proof}

\section{Limit and complete modules}
\subsection{Limit modules}
We introduce some special types of submodules of a module $\M$ which we call
\emph{limit modules}. These limit modules when further specialized to
zigzag modules exhibit additional properties, which we utilize in our algorithm.

\vspace{0.1in}
\begin{definition}[Limit module]
We call a submodule $\Int \subseteq \M$ a \emph{limit module} for $\M: P\rightarrow \vecsp$ if  
(i) $\Int=\oplus_i \Int^{I_i}$ where each $\Int^{I_i}$ is an interval module with $I_i\cap I_j=\emptyset$ for $i\neq j$
and (ii) for every pair $p\leq_{P} q$, the morphism $\Int (p\leq_{P} q)$ is a surjection (this allows some particular types of interval modules to be limit modules, see e.g. Eq.\eqref{eq:representative}). We call an ordered sequence of vectors $v_{p_0},\ldots,v_{p_m}$, $v_{p_i}\in \M(p_i)$, 
a \emph{limit representative} (or simply \emph{representative}) if there is
a limit module $\Int$ for $\M$ so that $\Int(p_i)$ is spanned by the vector $v_{p_i}$ for each $i\in\{0,\ldots,m\}$. Conversely, every limit module $\Int$ admits a representative
$\bs^{\Int}_{p_0},\ldots,\bs^{\Int}_{p_m}$ where $\bs^{\Int}(p_i)$ spans $\Int(p_i)$.
\label{def:limitmodule}
\end{definition}

In what follows we often say a limit module for $\M$ simply a limit
module when $\M$ is understood from the context.
The reader can observe that limit representatives are elements of $\limproj \, \M$.

The following observations about limit modules 
help understand the roles they play in the rest of the paper.

First, observe that a limit module $\Int$, in general, can be either a full interval module or a direct sum of
one or more non-overlapping interval modules such as the ones separated by the red arrows in~\eqref{eq:representative}. 

Second, observe that some of the interval modules
in a direct decomposition of $\M_{ZZ}=\bigoplus_i \Int_i$ may be limit modules.

Third, if $v_{p_0},\ldots, v_{p_m}$ is a representative of a limit
module $\Int$, then $\alpha v_{p_0},\ldots, \alpha v_{p_m}$ is also a 
representative of $\Int$ for any scalar $0\neq \alpha\in \mathbb{F}$. In regard
to this fact, we introduce the following notation.

\vspace{0.1in}
\noindent
\fbox{
\parbox{6.2in}{
   For a limit module $\Int\subseteq \M_{ZZ}$, if $\bs^{\Int}: \bs^{\Int}_{p_0},\ldots,\bs^{\Int}_{p_m}$ is a \emph{chosen} representative for $\Int$, then
   for $0\not=\alpha\in \mathbb{F}$,
   $\alpha\bs^{\Int}$ denotes the representative $\alpha \bs^{\Int}_{p_0},\ldots,\alpha\bs^{\Int}_{p_m}$.
    }
} 

\vspace{0.1in}
The reader may realize that a chosen representative of a limit module $\Int$
is an element in $\limproj \,\Int$ representing a global section of $\M_{ZZ}$. They can be added to produce other sections. This addition is given by pointwise vector addition which should not be confused with direct sums.
\begin{observation}[representative sums]
    For two limit modules $\Int$ and $\Int'$ and
    for $\alpha,\alpha'\in \mathbb{F}$, the sequence
    of vectors $(\alpha\bs^{\Int}_{p_0}+ \alpha'\bs^{\Int'}_{p_0}), \ldots,
    (\alpha\bs^{\Int}_{p_m}+ \alpha'\bs^{\Int'}_{p_m})$ is a representative. We denote this representative as the sum $\alpha\bs^{\Int}+\alpha'\bs^{\Int'}$.
    \label{obs:addition}
\end{observation}

\cancel{
\begin{definition}[Limit module sum]
    For $\alpha,\alpha'\in \mathbb{F}$,
    the sum $\alpha\Int+\alpha'\Int'$ for two limit modules
    $\Int$, $\Int'$ is a limit module with the fixed
    limit representative $(\alpha\bs_{\Int(p_0)}+\alpha'\bs_{\Int'(p_0)}),\ldots,(\alpha\bs_{\Int(p_m)}+\alpha'\bs_{\Int'(p_m)})$.
  \label{def:addition}
\end{definition}
}
The representative $\alpha\bs^{\Int}+\alpha'\bs^{\Int'}$ can be viewed
as an element in the space $\limproj\, \Int \oplus \limproj\,\Int'$ obtained
by fixing an element $\bs^\Int$ in $\limproj\, \Int$ and $\bs^{\Int'}$ in $\limproj\, \Int'$ and mapping them to the direct sum
by inclusions.

\subsection{Complete modules}
\label{sec:complete-limit}
\cancel{
Our aim is to unfold a $P$-module $\M$, $P$ being finite and connected, to a zigzag module 
$\M_{ZZ}$ defined over a zigzag poset $P_{ZZ}$ and then use Theorem~\ref{thm:full-folding} on a direct decomposition of $\M_{ZZ}$ to fold back some of its full interval modules.
Consider a direct decomposition $\M_{ZZ}=\bigoplus_i\Int_{i}$ of $\M_{ZZ}$ into
interval modules. 
Such a decomposition may not be unique because
different basis (representative)
vectors may be used to define the interval modules
over the same (multi)set of intervals. To apply Theorem~\ref{thm:full-folding}(1), the full interval modules in a decomposition of $\M_{ZZ}$ that we try to fold back
should themselves and have their complements be foldable. Our goal is to determine the maximum number of such full interval modules over all decompositions of $\M_{ZZ}$.
The following definition is introduced keeping this in mind. 
}

\begin{definition}
Let $P=\Fld_s Q$ 
and $\N: Q\rightarrow \vecsp$ be a $Q$-module so that
$\M=\Fld_s (\N)$ exists.
An interval module $\Int$ in a direct decomposition $\D:\bigoplus_i\Int_{i}$ of $\N$
is called \emph{$s$-complete} if and only if (i) $\Int$ is a full interval module that is foldable in $\D$ and (ii) its complement $\overline{\Int}$ in $\D$ is invertible. Let $\kappa(\D)$ denote the
number of $s$-complete interval modules in $\D$. We call
$\D$ \emph{$s$-complete} if $\kappa(\D)=\rk(\M)$.
\label{def:foldable}
\end{definition}

\cancel{
\begin{proposition} 
The number of $s$-complete intervals in any decomposition of $\M_{ZZ}$ cannot exceed $\rk(\M)$.
\label{prop:nomore}
\end{proposition}
\begin{proof}[Proof of Proposition~\ref{prop:nomore}]
Consider any decomposition of $\M_{ZZ}$ and let $r$ be the number of $s$-complete intervals in it. By definition, each $s$-complete interval module $\Int$ satisfies the requirement
in Theorem~\ref{thm:full-folding} and thus folds into a full interval summand of $\M$.
Then, by Theorem~\ref{thm:rk}, $r \leq \rk(\M)$.
\end{proof}
}
\cancel{
A simple fact about $s$-complete interval modules follows from the definition:
\begin{fact}
(ii) A full interval module $\Int$ in a direct decomposition of $\M_{ZZ}$ is $s$-complete iff 
$\Int(p)=\Int(p')$ and $\overline\Int(p)=\overline\Int(p')$ for every pair of points $p,p'\in P_{ZZ}$ where $s(p)=s(p')$.
\label{fact:fullsection}
\end{fact}
}

Theorem~\ref{thm:full-folding} helps us to prove the following proposition which guarantees that a foldable module $\N$ has a $s$-complete decomposition. Additionally, it states that no direct decomposition of $\N$ can have
more $s$-complete interval modules than a $s$-complete decomposition can have. 

\begin{proposition}
Let $P=\Fld_s Q$ be a folded poset of a finite poset $Q$.
Let $\N: Q\rightarrow \vecsp$ be a $Q$-module so that the $s$-folding
$\M=\Fld_s (\N)$ exists. Then, a $s$-complete decomposition of $\N$ exists. Furthermore, any direct decomposition $\D$ of $\N$ has $\kappa(\D)\leq \rk(\M)$.
\label{prop:complete-rank}
\end{proposition}
\begin{proof}
First, we prove the second conclusion. If a direct decomposition $\D$ of $\N$ had $\kappa(\D)>\rk(\M)$, then $\D$ would have more than $\rk(\M)$ $s$-complete interval modules as its summand each of which would fold to a full interval summand of $\M$ (Theorem~\ref{thm:full-folding}(3)). This is not possible because in that case $\M$ would have more than $\rk(\M)$ summands that are full intervals, an impossibility according to Theorem~\ref{thm:rk}.

Next, we show the first conclusion.
Consider a direct decomposition
$\M=\Int_1\oplus\cdots\oplus\Int_r\oplus \M'_1\oplus\cdots\oplus\M'_k$ where $\Int_1,\ldots,\Int_r$
are full interval modules. By Theorem~\ref{thm:rk}, $r=\rk(\M)$. Then, there is 
a decomposition $\D: \Fld_s^{-1}(\Int_1)\oplus\cdots\oplus\Fld_s^{-1}(\Int_r)\oplus \Fld_s^{-1}(\M'_1)\oplus\cdots\oplus\Fld_s^{-1}(\M'_k)$ of $\N$ by Theorem~\ref{thm:full-folding}(2). Furthermore, each of $\Fld_s^{-1}(\Int_i)$, $1\leq i\leq r$, is a full interval module because each $\Int_i$ is so. By definition, both $\Fld_s^{-1}(\Int_i)$ and its complement are foldable (hence also invertible). Therefore, each $\Fld_s^{-1}(\Int_i)$ is $s$-complete. A direct decomposition $\D'$ (into indecomposables) that refines $\D$
is an $s$-complete decomposition of $\N$ because it has $r=\rk(\M)$ $s$-complete
interval modules and it cannot have any more $s$-complete interval modules as $\kappa(\D)\leq \rk(\M)$. 
\end{proof}

\section{Zigzag module}
\label{sec:zigzag}
\begin{definition}
A poset $P$ is called a \emph{zigzag} poset iff there is a 
linear ordering $p_0,\ldots,p_m$
of the points in $P$, called the \emph{zigzag} path, so that 
for $i\in \{0,1,\ldots, m-1\}$, $p_i\leftrightarrow p_{i+1}$ are the only and all immediate pairs in $P$, i.e., zigzag path represents the Hasse diagram of $P$. We write $[p_i,p_j]$ to denote an interval
$I\subseteq P$ with the zigzag path $p_i,p_{i+1},\cdots, p_j$.
\end{definition}

\begin{definition}
A zigzag module $\M_{ZZ}: P_{ZZ}\rightarrow \vecsp$ is a persistence module where
the poset $P_{ZZ}$ is a zigzag poset. Assuming that 
$p_0,p_1,\ldots, p_m$ is the zigzag path for $P_{ZZ}$, we write the zigzag module as:
\begin{eqnarray}
\M_{ZZ}: V_{p_0}\stackrel{\phi_0}{\longleftrightarrow}\ldots \stackrel{\phi_{i-1}}{\longleftrightarrow}V_{p_i}\stackrel{\phi_i}{\longleftrightarrow} V_{p_{i+1}}\stackrel{\phi_{i+1}}{\longleftrightarrow}\ldots\stackrel{\phi_{m-1}}{\longleftrightarrow} V_{p_m}
\label{zigzag-module}
\end{eqnarray}
where for $i\in \{0,1,\ldots,m-1\}$, $V_{p_i}=\M_{ZZ}(p_i)$ denote the vector spaces and $\phi_i=\M_{ZZ}(p_i\leq_{P_{ZZ}} p_{i+1})$ or
$\M_{ZZ}(p_{i+1}\leq_{P_{ZZ}}p_i)$ denote
the morphisms (linear maps).
\end{definition}
We will be interested in interval submodules $\Int^{[b_i,d_i]}$ of $\M_{ZZ}$. Such an interval module is either full or can be of four types
determined by the types of its end points. The point $b_i$ is called
closed if $b_i\not=0$ and the arrow between $b_{i-1},b_i$ is a forward arrow `$\rightarrow$' and is called open otherwise. Similarly, the point $d_i$ is called closed if
$d_i\not=m$ and the arrow between $d_{i}, d_{i+1}$ is a backward arrow and is called open otherwise. The interval module $\Int^{[b_i,d_i]}$ is called open-open, open-closed, closed-open, or closed-closed depending on whether $b_i$ and $d_i$ are both open,
$b_i$ is open and $d_i$ closed, $b_i$ is closed and $d_i$ is open, or
both $b_i$ and $d_i$ are closed respectively. Notice that a full interval module
is considered open-open by definition.

By Theorem~\ref{thm:krull-schmidt}, a zigzag module $\M_{ZZ}:P_{ZZ}\rightarrow \vecsp$ over a  finite zigzag poset $P_{ZZ}$ also decomposes uniquely into indecomposables. By quiver theory~\cite{Gabriel72,Oudot15}, these indecomposables are interval modules, that is,
$\M_{ZZ}= \Int_1\oplus\cdots \oplus \Int_k$ where each $\Int_i:=\Int^{I_i}$ is an interval module
defined over an interval $I_i:=[p_{b_i},p_{d_i}]$. 

We will be interested in limit modules in direct decompositions of zigzag modules.
Observe that limit modules for a zigzag module $\M_{ZZ}$ are dierct
sum of open-open interval modules and take the following form.
\begin{eqnarray}
 0\leftrightarrow\cdots {\color{red}{0\leftarrow}} {\color{blue} \Int({p_{b_i}}) \leftrightarrow\cdots\leftrightarrow \Int({p_{d_i}})}{\color{red}\rightarrow 0} \leftrightarrow  \cdots \leftrightarrow {\color{red} 0 \leftarrow}{\color{blue} \Int({p_{b_j}}) \leftrightarrow\cdots\leftrightarrow \Int({p_{d_j}})}{\color{red}\rightarrow 0}\leftrightarrow \cdots \leftrightarrow 0
\label{eq:representative}
\end{eqnarray}

The following proposition says that in a sense any representative of
a full interval submodule of $\M_{ZZ}$ is in the
span of the representatives of the limit modules present in any direct decomposition of $\M_{ZZ}$. This property does not extend beyond zigzag module. The persistence
module $\M$ shown in Figure~\ref{fig:unfoldingEx3}(top left) has a submodule
which is a full interval module. However, it cannot be expressed as a direct sum
of any limit modules in a direct decomposition of $\M$ because $\M$ itself
is indecomposable and is not even an interval module. Also, see Conclusions (section~\ref{sec:conclusion})
for further comments on this aspect.
\begin{proposition}
    Let $L$ be the set of limit modules in a direct decomposition $\D$ of
    $\M_{ZZ}$. For any full interval module $\Int\subseteq \M_{ZZ}$,
    there exist unique $\alpha_i\in \mathbb{F}$ so that 
    $\bs^\Int=\sum_{\Int_i\in L} \alpha_i\bs^{\Int_i}$ where the sum is defined
    as in Observation~\ref{obs:addition}.
    \label{prop:limit-module}
\end{proposition}

\begin{proof}
    Let $\Int_1,\ldots,\Int_k$ be the set of all interval modules in the
    direct decomposition $\D$ of the $P_{ZZ}$-module $\M_{ZZ}$. For any $p\in P_{ZZ}$, since $\M_{ZZ}(p)=\text{span}(\bs^{\Int_1}_p,\cdots,\bs^{\Int_k}_p)$, there exist uniquely determined
    $\alpha^p_i\in\mathbb{F}$ so that
    $\bs^{\Int}_p=\sum_i \alpha^p_i \bs^{\Int_i}_{p}$ where
    $\alpha_i^p$ is taken to be zero if $\bs^{\Int_i}_{p}=0$. Let $L_p=\{\Int_i\,|\, \alpha^p_i\not=0\}$
    and $L'=\cup_p L_p$. It is not difficult to show that
    $\alpha^p_i=\alpha^q_i:=\alpha_i$ for any two points $p,q$
    in the support of $\Int_i$.
    Then, we can write $L'=\{\Int_i \,|\, \alpha_i\not = 0\}$. We claim that
    $L'$ is a subset of limit modules (open-open) in $\D$, that is, $L'\subseteq L$. If not, there is an interval module $\Int'\in L'$ that is not
    a limit module, that is, $\Int'$ has an end point $p_j\not \in \{p_0, p_m\}$ satisfying either of the following two cases: (i) the arrow for $\phi_{j-1}$ is forward and the cokernel of $\phi_{j-1}$ restricted to $\Int'(p_{j})$ is non-zero.
    It follows that the cokernel of $\phi_{j-1}$ restricted to $\Int(p_{j})$ is also non-zero. This
    is impossible as $\Int$ is a full interval module and $p_j\not\in\{p_0,p_m\}$, (ii) the arrow for $\phi_{j}$ is backward: again, we reach an impossibility with a similar argument. So, $L'\subseteq L$ and it follows that $\bs^\Int=\sum_{\Int_i\in L}\alpha_i\bs^{\Int_i}$ where $\alpha_i=0$ for $\Int_i\in L\setminus L'$ establishing the claim of the proposition. 
\end{proof}

\cancel{
    \begin{theorem}
        Let $\tilde M=\Fld_s M$ be a folded module of a $P$-module $M$ with a folding $s$.
        Two $\tilde M$-vectors $\tilde u$ and $\tilde v$ are indeed vectors
        for an indecomposable module of $M$ if and only if:
        \begin{itemize}
            \item there exist an $N$-vector $u$ and an $N'$-vector $v$ where $N$ and $N'$
            are indecomposable modules of $M$ so
            that $\tilde u=s(u)$ and $\tilde v=s(v)$, and either $N=N'$, or there
            are $N$- and $N'$-vectors $u'$ and $v'$ respectively with $s(u')=s(v')$.
        \end{itemize}
    \end{theorem}
\begin{theorem}
Let $\Fld_s M$ be a folded module of a $P$-module $M$ with a folding $s$. Let
$M=M_1\oplus\cdots\oplus M_k$ and $\Fld_s M=\tilde{M}_1\oplus\cdots\oplus\tilde{M}_\ell$
be the decompositions of $M$ and $\Fld_s M$ respectively where each $M_i$ and $\tilde{M}_i$
is an indecomposable module. Then, $\forall i\in\{1,\ldots,\ell\}$,
$$
\tilde{M}_i= \Fld_s (M_{i_1}\oplus\cdots\oplus M_{i_t}) \mbox{ where }
\{i_1,\ldots,i_t\}\subseteq  \{1,\ldots, k\}.
$$
\end{theorem}
}

Proposition~\ref{prop:limit-module} suggests the following approach 
to compute the generalized rank of a given $P$-module $\M$:
first unfold $P$ into a zigzag poset (path) $P_{ZZ}$ and construct a zigzag module
$\M_{ZZ}$ that is an $s$-unfolding of $\M$. It follows that
$\Fld_s(\M_{ZZ})$ exists and $\M=\Fld_s(\M_{ZZ})$ as required
by Proposition~\ref{prop:limit-module}. Then, after computing a direct decomposition
of $\M_{ZZ}$ with a zigzag persistence algorithm, convert it to an $s$-complete decomposition
and determine how many 
full interval modules (if any) in this decomposition
are $s$-complete. 

\begin{figure}[htbp]
\centerline{\includegraphics{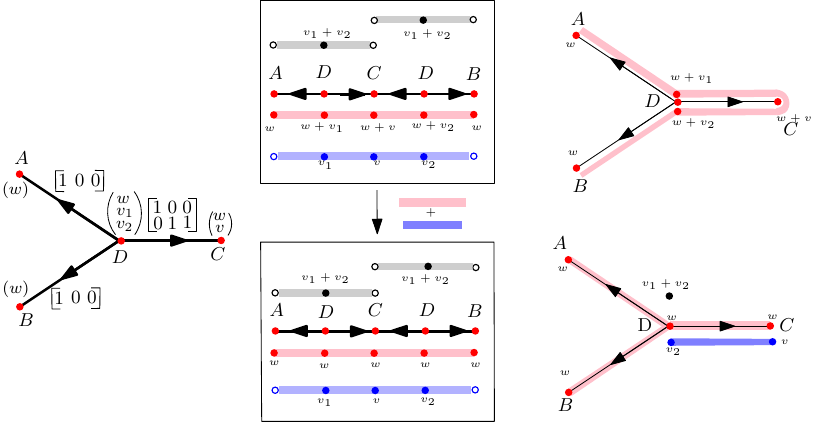}}
\caption{The full red interval module in $\M_{ZZ}$ as shown in top-middle does not
fold to a full interval submodule of $\M$ because the vectors at $D$ mismatch (top-right), but after adding the blue interval module
(limit module), it is converted to a full interval module as shown in bottom-middle
that folds back to a full interval module (pink) that is a summand of $\M$ (bottom-right).}
\label{fig:unfoldingEx2}
\end{figure}
Consider the $P$-module $\M$ shown in Figure~\ref{fig:unfoldingEx2} (left). After unfolding
the module to a zigzag module (middle), suppose we get a decomposition into
interval modules as indicated in the middle-top picture. Just like the example
in Figure~\ref{fig:unfoldingEx} (bottom), the full interval module in this decomposition of $\M_{ZZ}$
does not fold into a full interval submodule of $\M$ because the representative vectors at the two copies
of $D$ do not match. In the example of Figure~\ref{fig:unfoldingEx}, we could not
repair this deficiency. However, now we can do so using the limit modules. 
Observe that the open-open interval module (blue) supported
on $D\rightarrow C \leftarrow D$ is
a limit module. We can add its representative
to the representative of the full interval module to obtain a new
full interval module shown in middle-bottom picture. This new
module is complete because it and its complement
are foldable and thus the module folds into a full interval summand of $\M$ (Theorem~\ref{thm:full-folding}(1)).
Observe that any of the other two limit modules (grey) could also serve the purpose. 
The algorithm {\sc GenRank} in section~\ref{sec:algorithm} essentially determines whether
a full interval module $\Int$ in the current decomposition of $\M_{ZZ}$ is foldable,
and if not, whether it can be \emph{converted} to one by adding the chosen representatives of a set of limit modules. This is followed up by a check that determines if the complement of the
module $\Int$ is invertible. If so, it marks converted $\Int$ to be complete.

It is instructive to point out that the blue module which is not a full module can also be made foldable
by adding the representative of one of the grey modules, say the left one, to its representative, which has been done to obtain the decomposition shown on bottom right in Figure~\ref{fig:unfoldingEx2}.
Our algorithm does not do this because we are interested only on folding full interval modules.

\subsection{Unfolding to a zigzag path and zigzag module}
\label{sec:unfold}
A finite poset $P$ is represented by a directed acyclic graph $G=(P,E(P))$ where (i) every directed edge $(p,q)\in E(P)$ satisfies 
$p\leq_P q$ and (ii) 
every immediate pair $p\rightarrow q$ in $P$ must correspond to an edge
$(p,q)\in E(P)$. The size of the poset $|P|$ with such a 
representation is measured as the total number of vertices and edges in $G$.
Given a directed graph $G=(P,E(P))$ for a finite poset $P$, we construct a zigzag poset $P_{ZZ}$ using the concept of
Eulerian tour in graphs so that $G_{ZZ}=(P_{ZZ},E(P_{ZZ}))$ 
represents the zigzag path for $P_{ZZ}$ 
and $P=\Fld_s P_{ZZ}$ for some folding $s$.
\begin{figure}[htbp]
\centerline{\includegraphics{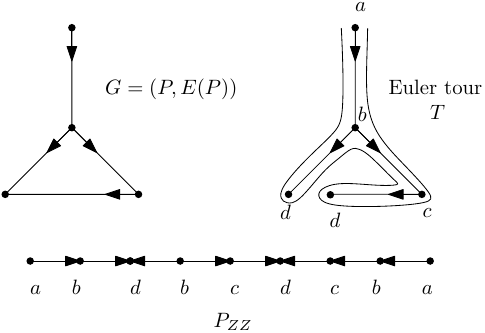}}
\caption{Unfolding $P$ (left) to a zigzag poset $P_{ZZ}$; wrapping a thread
around (right) gives a zigzag path (bottom) in $P_{ZZ}$}
\label{fig:unfoldingPath}
\end{figure}
Given a connected graph $G=(V,E)$, an Eulerian tour in $G$ is an ordered sequence of its
vertices, possibly with repetitions, $u_0,\ldots, u_i,u_{i+1},\ldots u_t=u_0$ so that every edge $(p,q)\in E$ appears exactly once as a consecutive pair $(p=u_i,u_{i+1}=q)$ of vertices in the sequence. It is known that if $G$ has even degree at every vertex, then $G$ necessarily has an Eulerian tour which can be computed
in $O(|V|+|E|)$ time. We consider the undirected version $\overline{G}=(P,E(P))$ of the poset graph $G=(P,E(P))$ and
straighten it up using an Eulerian tour, see Figure~\ref{fig:unfoldingPath}. However, the graph $\overline{G}$ may not satisfy the
vertex degree requirement. So, we double every edge, that is, put a parallel edge
in the graph for every edge (this is equivalent to wrapping a thread around as shown in Figure~\ref{fig:unfoldingPath}). The modified graph $\overline{G}$ then has only even-degree vertices. 
We compute an Eulerian tour $T$ in $\overline{G}$ and for every adjacent pair of 
vertices $p,q$ in the tour representing
an edge $(p,q)$ in $\overline{G}$ we impose the order $p\leq_T q$ if and only if
the directed edge $(p,q)$ is in $G$. The poset $(T,\leq_T)$ is taken as the zigzag poset
$P_{ZZ}$ and the tour (zigzag path) as its representation. Clearly, the number of
edges in the tour (immediate pairs in $P_{ZZ}$) is at most twice the number of edges in $G$ and the number of vertices in $P_{ZZ}$ is one more than that number. We have
\begin{fact}
    $P=\Fld_s P_{ZZ}$ for some folding $s$ where $|P_{ZZ}| \leq 2 |P|+1$.
    \label{cor:folding}
\end{fact}

In our algorithm, we assume 
that $\M$ is implicitly given by a $P$-filtration:
A $P$-filtration $F(K)$ of a simplicial complex $K$ is a family of
subcomplexes $F(K)=\{K_p\subseteq K\}_{p\in P}$ so that
$K_p \subseteq K_q$ if $p\leq_P q$. We assume that both $K$ and $P$ are finite.
Applying the homology functor $H_k(\cdot)$ to the filtration $F(K)$, one obtains a module $\M:=\M_{F(K)}$
in degree $k$ where $\M(p)=H_k(K_p)$ and $\M(p\leq_P q):H_k(K_p)\rightarrow H_k(K_q)$ induced by the inclusion $K_p\subseteq K_q$. 

First, we unfold the poset $P$ into a zigzag poset $P_{ZZ}$ using the method described
before. Let $s$ be the resulting folding given by
Fact~\ref{cor:folding}. To unfold $\M: P\rightarrow \vecsp$ into a zigzag module $\M_{ZZ}: P_{ZZ}\rightarrow \vecsp$, we build a zigzag filtration $F_{ZZ}=\{K_p\}_{p\in P_{ZZ}}$
by assigning $K_p:=K_{s(p)}$. To check that $F_{ZZ}$ is indeed a zigzag filtration, observe that
$K_p\subseteq K_q$ for every $p\leq_{P_{ZZ}} q$ because (i) $s(p)\leq_P s(q)$ by definition of
folding $s$ and (ii) $K_{s(p)}\subseteq K_{s(q)}$ by definition of the filtration $F(K)$. It can be easily verified that applying the homology functor on 
$F_{ZZ}$, we get the $s$-unfolding $\M_{ZZ}: P_{ZZ}\rightarrow \vecsp$ of
$\M$.


\section{Algorithm}
\label{sec:algorithm}
The algorithm ({\sc GenRank} in pseudocode) takes a $P$-fitration $F$ and
a degree $k$ for the homology group.
First, it $s$-unfolds $P$ to a zigzag path of $P_{ZZ}$ and computes the filtration $F_{ZZ}$ (Step 1).
Let $\M$ and $\M_{ZZ}$ be the modules obtained by applying the
homology functor in degree $k$ on $F$ and $F_{ZZ}$ respectively 
as described above.
We need to compute interval modules from $F_{ZZ}$ that represents
a direct decomposition of $\M_{ZZ}$, i.e., we need a zigzag persistence algorithm that
computes the intervals in the barcode 
with a \emph{representative} (step 3).
A sequence of $k$-cycles $z_b,\ldots,z_d$ constitutes a representative
of a limit module $\Int:=\Int^{[p_b,p_d]}$ if $[z_i]$ is chosen as $\bs^{\Int}_{p_i}$ for $p_i\in [p_b,p_d]$. The zigzag persistence algorithm in~\cite{DHM25} can compute these representatives efficiently.
Next, the algorithm checks how many interval modules in the computed decomposition of $\M_{ZZ}$ can be converted to $s$-complete modules
which provides
$\rk(\M)$ according to the definition of $s$-completeness. 

Next proposition tells us that it is sufficient to check only the full
interval modules in a direct decomposition of $\M_{ZZ}$ if they can be converted to $s$-complete modules.

\begin{proposition}
Let $\M_{ZZ}$ be an $s$-unfolding of $\M$ and $\Int_1,\ldots,\Int_\ell$ be the set of limit modules in a direct decomposition $\D$ of $\M_{ZZ}$. 
There exist unique $\alpha_i\in\mathbb{F}$, $i\in [\ell]$, so that every
$s$-complete interval module $\Int$ in any direct decomposition $\D'$ of $\M_{ZZ}$ satisfies that
$\bs^\Int=\sum_{i=1}^\ell\alpha_i\bs^{\Int_i}$ where for at least one $i\in [\ell]$, $\Int_i$ is a full interval module and $\alpha_i\not=0$.
\label{prop:complete-interval}
\end{proposition}
\begin{proof}
First, Proposition~\ref{prop:limit-module} allows us to write $$\bs^\Int=\sum_{i=1}^\ell\alpha_i\bs^{\Int_i} \mbox{ for
some unique }\alpha_i\in\mathbb{F}, i\in [\ell].
$$
Suppose that the claim of the proposition is not true. 
Fix a point $p\in P_{ZZ}$.
We have $\bs^{\Int}_p=\sum_{i=1}^\ell\alpha_i \bs^{\Int_i}_p$. 
Recall the quotient map $\rho$ for
colimit in Proposition~\ref{prop:alternate-limit}\ref{item:colim}.
For each vector
$v_i=\bs^{\Int_i}_p$, the quotient vector $\rho(v_i)$ is a zero element in the colimit $\colim\, \M_{ZZ}$ because the limit module $\Int_i$ which is not
full either has a sequence of vectors $v_i\leftrightarrow \cdots\rightarrow 0$ or $0\leftarrow \cdots \leftrightarrow v_i$ in its representative and thus
$v_i\sim 0$ (Notation~\ref{not:sim}). It follows that any representative of
$\Int$ is sent to a zero element
in $\colim\, \M_{ZZ}$. Since $\Int$ is $s$-complete, $\Fld_s(\Int)$ exists and its
representative is an element of $\limproj\, \M$. The definition
of the folding implies that the limit-to-colimit
map sends the representative of $\Fld_s(\Int)$ also to a zero element in $\colim\, \M$.


Since $\Int$ is $s$-complete, it is foldable and its complement is also foldable. Then, $\Fld_s(\Int)$ is a full interval summand of $\M$ according to
Theorem~\ref{thm:full-folding}(1). Since the full interval summand $\Fld_s(\Int)$ is sent to zero by the limit-to-colimit map, we have 
$\rk(\M')=\rk(\M)$ where $\M=\Fld_s(\Int)\oplus \M'$. However, $\rk(\M')\leq \rk(\M)-1$
according to Theorem~\ref{thm:rk} reaching a contradiction.
\end{proof}

The above proposition suggests that we try to \emph{convert} every full interval module $\Int$ in a direct decomposition $\D$ of $\M_{ZZ}$ to a foldable module first, and then appealing
to Theorem~\ref{thm:full-folding}(3) check if the complement module
$\overline{\Int}$ in $\D$ is invertible.

\begin{definition}
    Let $\M_{ZZ}$ be an $s$-unfolding of $\M$.
    We say a full interval module $\Int$ in a direct decomposition $\D$ of $\M_{ZZ}$ is \emph{convertible} in $\D$ if either (i) $\Int$ is foldable, or (ii) there exists a set of limit modules 
    $\{\Int_i\}$ in $\D$ none of which is equal to $\Int$ so that $\Int'$ with
    representative $\bs^{\Int'}:=\bs^\Int + \sum_i \alpha_i\bs^{\Int_i}$, $0\not=\alpha_i\in \mathbb{F}$, is foldable.
    \label{def:convertible}
\end{definition}
Earlier, we defined $\kappa(\D)$ as the number of $s$-complete interval modules
in a decomposition $\D$ of a zigzag module $\M_{ZZ}$. Recall that some
of the full interval modules in $\D$ may not be $s$-complete but can be converted
to be $s$-complete. To account for those, we introduce the following notation.
\begin{notation}
    Let $\M_{ZZ}$ be an $s$-unfolding of $\M$ and $\D$ be any of its direct decomposition. Denote by $\tau(\D)$ the number of interval modules $\Int$ in $\D$ such that $\Int$ is convertible and $\bar{\Int}$ is invertible in $\D$. Observe that $\kappa(\D)\leq \tau(\D)$ because $s$-complete interval modules are convertible by condition (i) in Definition~\ref{def:convertible} and their complements are invertible by definition.
\end{notation}

\begin{proposition}
    $\kappa(\D)\leq \rk(\M) \leq \tau(\D)$.
    \label{prop:convert}
\end{proposition}
\begin{proof}
  It follows from Proposition~\ref{prop:complete-rank} that $\kappa(\D)\leq \rk(\M)$.
  So, we only show $\rk(\M)\leq \tau(\D)$. 
  Observe that if $\rk(\M)=0$ there is nothing to prove since $\tau(\D)$ is non-negative by definition, so assume $\rk(\M)\not=0$.
  Consider an $s$-complete decomposition $\D^*$ of $\M_{ZZ}$ which exists according to Proposition~\ref{prop:complete-rank}. Let $\Jnt_1,\ldots,\Jnt_r$, $r=\rk(\M)$, denote the set of these $s$-complete modules in $\D^*$. We show by induction that there is a set of full modules $\Int_1,\ldots,\Int_r$ in the direct decomposition $\D$ of $\M_{ZZ}$ so that $\Int_i$ is convertible to $\Jnt_i$ and $\bar{\Int_i}$ invertible,
  $i\in [1,r]$ establishing the claim.


  For the base case, consider $\Jnt_1$. By Proposition~\ref{prop:complete-interval}, there is a set of limit modules $\Int'_i$ in $\D$ so that
  there exist unique $0\not=\alpha_i\in \mathbb{F}$ giving $\bs^{\Jnt_1}=\sum_i \alpha_i\bs^{\Int'_i}$. Choose any full module among $\Int_i'$s as $\Int_1$ which is guaranteed to exist by Proposition~\ref{prop:complete-interval}. Then, consider the decomposition $\D_1$ that only replaces $\Int_1$ with $\Jnt_1$ in $\D_0=\D$. 
  Observe that $\M_{ZZ}=\Int_1\oplus\overline{\Int_1}=\Jnt_1\oplus\overline{\Jnt_1}=\Int_1\oplus\overline{\Jnt_1}$
  establishing that $\overline{\Int_1}\cong\overline{\Jnt_1}$ where $\overline{\Jnt_1}$ is invertible (actually foldable). This means
  $\overline{\Int_1}$ is invertible as required.
  
  \cancel{
  Observe that, $\M_{ZZ}=\Jnt_1\oplus\overline{\Int_1}$. Applying Theorem~\ref{thm:full-folding}(3), we get that $\Fld_s(\overline{\Int_1})$ exists
  because $\Fld_s(\Jnt_1)$ exists and is a summand of $\M$ by definition of $s$-completeness and
  Theorem~\ref{thm:full-folding}(1). We conclude that $\overline{\Int_1}$ is foldable.}

  
  To complete the induction, assume that, for $1\leq j<r$, we already have that $\Int_1,\ldots,\Int_j$ are convertible to $\Jnt_1,\ldots,\Jnt_j$ and $\overline{\Int_1},\ldots,\overline{\Int_j}$ are invertible.
  Also, we have the decompositions $\D_0,\ldots, \D_j$ of $\M_{ZZ}$ where $\D_i$ is obtained inductively from $\D_{i-1}$ by replacing $\Int_{i}$ with $\Jnt_{i}$.
  By Proposition~\ref{prop:complete-interval}, we have $\bs^{\Jnt_{j+1}}=\sum_i\alpha_i \bs^{\Int'_i}$, $\alpha_i\in \mathbb{F}$, for some limit modules in $\D_j$. We claim that the collection $\{\Int'_i\}$ includes a full module other than $\Jnt_1,\ldots,\Jnt_j$. If not, we have
  $\bs^{\Jnt_{j+1}}=(\sum_i\gamma_i\bs^{\Jnt_i})+(\sum_k\gamma'_k\bs^{\Int'_k})$ for
  $\gamma_i,\gamma'_k\in \mathbb{F}$ where none of $\{\Int'_k\}$ is full.
  Then, $\bs^{\Jnt_{j+1}}-\sum_i\gamma_i\bs^{\Jnt_i}=\sum_k\gamma'_k\bs^{\Int'_k}$.
  The LHS gives a representative of an $s$-complete module which is a sum of 
  representatives of limit modules(RHS)
  none of which is full contradicting
  Proposition~\ref{prop:complete-interval}. Thus, the set $\{\Int'_i\}$ includes a full module other than $\Jnt_1,\ldots,\Jnt_j$. Let $\Int_{j+1}$ in $\D_j$ be any such full module. Observe that $\Int_{j+1}$ is also in $\D$. It follows that $\Int_{j+1}$ in $\D$ is convertible to $\Jnt_{j+1}$. 
  Also, with the same reasoning as in the base case, one can check that $\overline{\Int_{j+1}}$ in $\D$ is invertible.
\end{proof}
\begin{theorem}
Let $\M_{ZZ}$ be an $s$-unfolding of $\M$ and $\D$ be any direct decomposition of $\M_{ZZ}$. 
\begin{enumerate}
    \item If every convertible full module $\Int$ in $\D$ where $\overline\Int$ is invertible
    is $s$-complete, 
    then $\kappa(\D)=\rk(\M)$ ($\D$ is $s$-complete) else
    \item Let $\Int$ be a convertible module where $\overline\Int$ in $\D$ is invertible. Let
    $\D'$ be the direct decomposition obtained from $\D$ by
    replacing $\Int$
    with the converted module
    $($$\mathrm{Convert}(\Int)$ in step 4 of {\sc GenRank}$)$, then $\kappa(\D')=\kappa(\D)+1$.
\end{enumerate}
\label{thm:algo}
\end{theorem}
\begin{proof}
    For (1), observe that, $\kappa(\D)=\tau(\D)$ in this case implying $\kappa(\D)=\rk(\M)=\tau(\D)$ due to Proposition~\ref{prop:convert}.

    For (2), observe that an interval module $\Jnt\neq\Int$ in $\D$ is foldable iff it is foldable in $\D'$ because the only affected module in $\D$ is $\Int$. 
    We observe that $\overline\Jnt$ remains invertible in $\D'$ if and only if it were
    so in $\D$. 
    It follows that, compared to $\D$, the decomposition
    $\D'$ has exactly one more
    foldable module, namely $\mathrm{Convert}(\Int)$, with its complement being invertible. Hence $\kappa(\D')=\kappa(\D)+1$. 
    \end{proof}
   \cancel{ 
    To prove the claim, observe that for any two points $p$ and $p'$ in $P_{ZZ}$ with $s(p)=s(p')$, $\overline\Jnt(p)=\overline\Jnt(p')$ in $\D$ because $\Jnt$ is assumed to be foldable. Since $\overline\Jnt(p)=\mathrm{span}(\bs^{\Int_1}_p,\ldots,\bs^{\Int_s}_{p})$ where $\Int_1,\ldots,\Int_s$ is the set of interval modules in $\D$ other than $\Jnt$, we have
    $\overline\Jnt(p)=\mathrm{span}(\bs^{\Int_1}_p,\ldots,\bs^{\Int_s}_p)$
    and $\overline\Jnt(p')=\mathrm{span}(\bs^{\Int_1}_{p'},\ldots,\bs^{\Int_s}_{p'})$. After converting
    $\Int$, these spans can change only if
    the vector $\bs^{\Int}_p$ changes due to the addition of the vector $\bs^{\Jnt}_p$ and the vector $\bs^{\Int}_{p'}$ changes due to the addition of the vector $\bs^{\Jnt}_{p'}$. Since $\Jnt$ is foldable in $\D$, we have $\bs^{\Jnt}_{p}=\bs^{\Jnt}_{p'}$ and thus the new spans of the basis vectors
    at $p$ and $p'$ remain equal in $\D'$ if and only if they were so in $\D$. 
 }
The algorithm {\sc GenRank} draws upon Theorem~\ref{thm:algo}. To accommodate
computations with a fixed finite precision, we assume that the field
$\mathbb{F}$ is finite.
The algorithm takes every full interval module $\Int$ that is not foldable and checks if it is convertible. If it is convertible, it converts it and then check if its complement $\overline\Int$ for invertibility. It continues
converting such modules to foldable modules until it cannot find any to convert.
The current direct decomposition $\D$ of $\M_{ZZ}$ changes with these conversions.
At the end of this process, all convertible modules in the current
decomposition whose complements are invertible become foldable themselves.
Their number then coincides with $\rk(\M)$.
To determine the existence of the limit modules whose addition makes $\Int$
foldable, we take the help of an annotation matrix $A_p$ ~\cite[Chaper 4]{DW22} 
for each complex $K_p$, $p\in P_{ZZ}$, computed in step 2.
Without further elaborations, we only mention that annotation
for a $k$-cycle $z\in Z_k(K_p)$ (cycle group in degree $k$) 
is the coordinate of its class $[z]$
in a chosen basis of $H_k(K_p)$; see~\cite{BCCDW12} (annotation algorithm
in~\cite{BCCDW12} is described for $\mathbb{F}=\Z_2$ though it easily extends
for any finite field $\mathbb{F}$).
The representative cycles maintained for every interval module
in the decomposition of $\M_{ZZ}$ at point $p$ form a cycle basis $[z_1],\ldots,[z_g]$ 
of $H_k(K_p)$. We will see later
that testing a full module for foldability amounts to
calculating the annotations of the representative cycles of limit
modules for certain points $p\in P_{ZZ}$ and performing certain matrix
reductions with them; see section~\ref{sec:convert}.\\

\noindent
{\bf Algorithm} {\sc GenRank} ($P$-filtration $F$, $k\geq 0$)\label{algorithm:GenRank}
\begin{itemize}
    \item Step 1. Unfold $P$ and $F$ into a zigzag path $P_{ZZ}$ and a
    zigzag filtration $F_{ZZ}$ respectively
    \item Step 2. Compute an annotation matrix $A_p$ for every complex $K_p$ in $F_{ZZ}$, $p\in P_{ZZ}$
    \item Step 3. Compute a barcode for $F_{ZZ}$ with representative $k$-cycles; Let $\mathcal I$ denote the set of full interval modules  corresponding to full bars in the current decomposition $\D$
    \item Step 4. For every module $\Int\in {\mathcal I}$ do
        \begin{itemize}
           \item If $\Int$ is convertible in $\D$, update $\D$ with $\Int\leftarrow\mathrm{Convert}(\Int)$
        \item If $\overline{\Int}$ is invertible in $\D$, mark $\Int$ \texttt{complete} 
        \end{itemize}       
    \item Output the number of \texttt{complete} interval modules/*may output converted modules*/
\end{itemize}

Now we focus on the crucial checks in step 4. Fix the degree of all homology groups
to be $k\geq 0$, which form the vector spaces of the modules in our discussion.

\subsection{Convertibility of $\Int$ and computing $\mathrm{Convert}(\Int)$}
\label{sec:convert}
Assume that $\Int$ is a full interval module in the current direct decomposition $\D$ of $\M_{ZZ}=\bigoplus_i \Int_i=\bigoplus_i \Int^{[p_{b_i},p_{d_i}]}$.
For each interval module $\Int_i=\Int^{[p_{b_i},p_{d_i}]}$ the algorithm computes a sequence of
\emph{representative} $k$-cycles $\{z^i_{p_j}\in Z^k(K_{p_j}) \mid p_j\in \{p_{b_i},p_{b_i+1},\ldots, p_{d_i}\}\subseteq P_{ZZ}\}$.
A representative
for a limit module $\Int^{[p_{b_i},p_{d_i}]}$ 
is given by the homology classes
$[z^i_{p_{b_i}}],\ldots, [z^i_{p_{d_i}}]$. Such a representative
for each limit module could be computed
by an adaptation of the zigzag persistence algorithm in~\cite{maria2014zigzag} or in \cite[Chapter 4,section 4.3]{DW22}. Instead, we apply the more efficient algorithm 
given in~\cite{DHM25}.
Then, a conversion of a convertible
module $\Int$ updates this representative as we update $\Int$ to $\mathrm{Convert}(\Int)$.


For any point $p\in P_{ZZ}$, call the set of all points that fold to $s(p)$ its partners and denote it as $\prt(p)$. We check if there are limit modules whose representatives when added
to the representative of $\Int$ makes the vectors at each point in $\prt(p)$ the same for every point $p\in P_{ZZ}$ because that converts $\Int$ to a foldable interval module. The partner sets partition $P_{ZZ}$. We designate an arbitrary point, say $p$, in each partner set $\prt(p)$ where $|\prt(p)|>1$ as the \emph{leader} of $\prt(p)$. Let
$L$ denote the set of these leaders. For every point $p\in L$,
we do the following.

Let $\Int_1,\ldots,\Int_\ell$ denote the set of limit 
modules in the current decomposition $\D$ of $\M_{ZZ}$ and WLOG assume $\Int=\Int_1$.
The module $\Int$ is convertible iff for all $p\in L$, there exist 
$\alpha_i\in\mathbb{F}$ so that
$\bs^{\Int_1}_p + \sum_{i=2}^\ell\alpha_i\bs^{\Int_i}_p=
(\bs^{\Int_1}_{p_j} + \sum_{i=2}^\ell \alpha_i\bs^{\Int_i}_{p_j})$ for every $p_j\in\prt(p)$. Written differently, for all $p\in L$ and for every $p_j\in \prt(p)$,
we require
\begin{equation}
 \bs^{\Int_1}_p+\bs^{\Int_1}_{p_j}=\sum_{i=2}^\ell\alpha_i(\bs^{\Int_i}_{p}+\bs^{\Int_i}_{p_j})=\sum_{i=2}^\ell\alpha_iv^i_j.
\label{eq:convertible}
\end{equation}
where $v^i_j:=\bs^{\Int_i}_{p}+\bs^{\Int_i}_{p_j}$.
\begin{figure}[htbp]
\centerline{\includegraphics[width=0.8\textwidth]{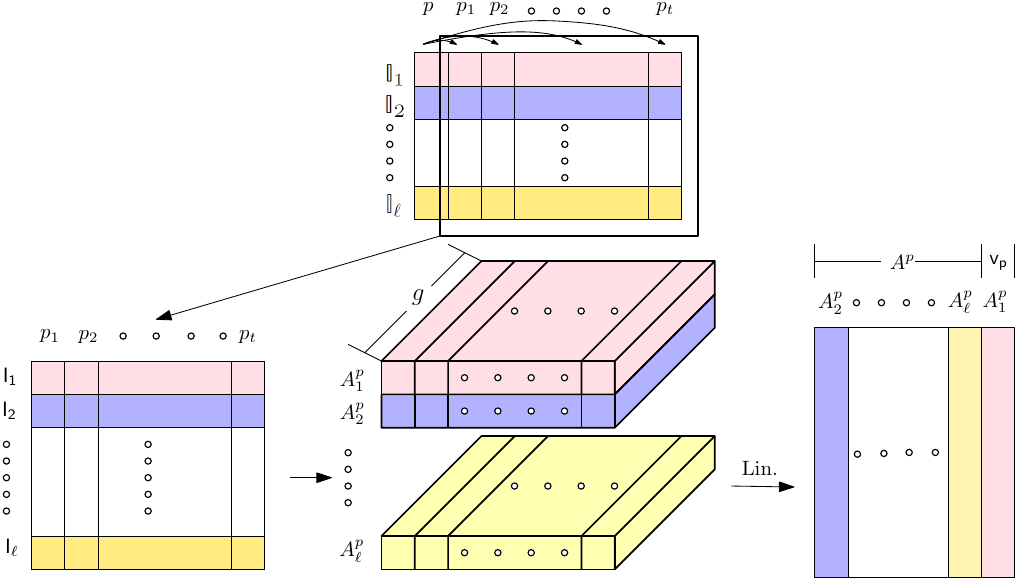}}
\caption{Illustration of matrix reduction with linearization to check Eq.~\eqref{eq:convertible}; (top) $p$ is representative of
$\prt(p)=\{p\}\cup\{p_1,\ldots,p_t\}$, each row $\Int_i$ is a limit
module with annotations of representatives at points in $\prt(p)$; (bottom left) row $\sf I_i$ consists of annotations of
cycles $\{\hat{z}^i_{p_j}\}_j$ where $\hat{z}^i_{p_j}=z^i_p+z^i_{p_j}$; (bottom middle) annotations of cycles in row $\sf I_i$ form matrix $A_i^p$ in the stacks of matrices; (bottom right) each matrix $A_i^p$, $i=2,\ldots,\ell$, linearized into a vector and combined into a larger matrix $A^p$ and $A_1^p$ into vector $\sf v_p$.}
\label{fig:MatrixReduct}
\end{figure}
Our goal is to determine coefficients $\alpha_i$s if Eq.~\eqref{eq:convertible} holds. To do this, for every $i=1,\ldots, \ell$ and every $p\in L$,
consider the matrices $A_i^p$ whose columns are annotations of the representative cycles
$\hat{z}_{p_j}^i$ of $v^i_j$ for every $p_j\in \prt(p)$; see Figure~\ref{fig:MatrixReduct} for an illustration. 
We can compute the 
$\mathbb{F}$-sum of two chains
$\hat{z}_{p_j}^i=z_{p_j}^i+z_p^i$ and add their annotation vectors
which is well defined because their
annotation vectors have the same dimension $g_p=\dim H_k(K_p)$
since all complexes at $p$ and $p_j$s were made equal during unfolding. 
Notice that, if $t_p=|\prt(p)|$, the matrix
$A_i^p$ has dimensions $g_p\times t_p$.
Then, checking Eq.~\eqref{eq:convertible} boils down to
determining $\alpha_i$s so that $A_1^p=\sum_{i=2}^\ell\alpha_iA_i^p$. This
can be done by using the linearization trick presented in~\cite{DX22}. Each of the
$g_p\times t_p$ matrices $A_i^p$, $i\in \{2,\ldots,\ell\}$,
is linearized into a vector $\sf v_i$ of length $t_pg_p$
and concatenated into a matrix of dimensions $t_pg_p \times \ell$.
Let $A^p$ denote this matrix (Figure~\ref{fig:MatrixReduct},bottom-right).

 Suppose that
there are $t=\sum_{p\in L} t_p$ points in $\cup_p \prt(p)$ and $g=\max_p\{g_p\}$.
We need to check Eq.~\eqref{eq:convertible} simultaneously for all $p\in L$ to determine the 
$\alpha_i$s. To do this, we concatenate matrices $A^p$ for all $p\in L$ each of
dimensions $t_pg_p\times \ell$ to create a matrix $A$ of dimensions $(\sum_pg_pt_p)\times \ell$. 
Then, a vector $\sf v$ of dimension $\sum_p g_pt_p$ is created by concatenating vectors $\sf v_p$ of dimension
$t_pg_p$ obtained by linearizing the matrices $A_1^p$ for all $p\in L$.
Checking if Eq.~\eqref{eq:convertible} holds simultaneously for all $p\in L$ boils down to
checking if $\sf v$ is in the column space of $A$. This is a matrix rank
computation on a matrix of dimensions $(\sum_pg_pt_p)\times \ell$ which can be done in $O(\ell^{\omega-1}\sum_pg_pt_p)=O(\ell^{\omega-1}gt)$ time where $\omega< 2.373$ is the
matrix multiplication exponent~\cite{JPS13}.

If Eq.~\eqref{eq:convertible} holds, we need to compute $\alpha=[\alpha_2,\ldots,\alpha_\ell]^T$ which can be done by
solving the linear system $A\alpha={\sf v}$. Then, we update $\Int$ to $\mathrm{Convert}(\Int)$ with the representative $\bs^\Int + \sum_{i=2}^\ell \alpha_i\bs^{\Int_i}$ to get the new decomposition $\D$. This takes $O(\ell^{\omega-1} gt)$ time again.

Notice that $g = O(n) = O(t)$.
There are $O(t)$ interval modules in the computed decomposition of 
$\M_{ZZ}$ and there are at most $O(n)$
full interval modules among them. So, we check at most $O(n)$
full interval modules for convertibility. If there are $\ell = O(t)$
number of limit modules, each convertibility check takes
$O(\ell^{\omega-1}tg) = O(t^\omega n)$ time giving a total of $O(t^\omega n^2)$ time for converting all convertible interval modules in step 4.

To finish the analysis of step 4, we need to consider the time spent for creating the
annotations for cycle representatives of the limit modules. For a representative cycle $z_p$, we get its annotation in the complex $K_p$ by adding the annotations of length
$g_p=\dim H_k(K_p)$ for every simplex in $z_p$. Thus, to compute the annotations
of all $g_p$ representative cycles at point $p$ (classes of non-trivial representative cycles at $p$ form a basis of $H_k(K_p)$), we consider a $g_p\times n_p$ matrix $E$ whose columns are annotations for each of the $n_p$ $k$-simplices in $K_p$
and another matrix $G$ of dimensions $n_p\times g_p$ whose columns represent
each cycle representative in the chain space of $K_p$. The product matrix $EG$
provides the annotations for each of the cycle representatives. This takes
$O(n_pg_p^{\omega-1})=O(ng^{\omega-1})$ time. Accouting for all $p\in P_{ZZ}$ and $g=O(n)$, we obtain a complexity of $O(tn^{\omega})$. Then, over all iterations
of convertibility, we get a time $O(tn^{\omega+1})$.

\subsection{Invertibility of $\overline{\Int}$}
To check invertibility of $\overline{\Int}$, we do the following: first, we observe that $\Fld_s(\Int)$ is well defined because $\Int$ is already converted to be foldable.
Consider the quotient module $\Q=\M/(\Fld_s(\Int))$. If $\Q$ is a submodule
of $\M$, we know that $\Fld_s(\Int)$ is a summand. In other words, $\overline{\Int}$
is then invertible due to Theorem~\ref{thm:full-folding}(3). We construct a $P$-filtration $\hat{F}$ that induces $\Q$ and check algorithmically if it is a submodule of $\M$ as follows.

After converting $\Int$ to a foldable module, suppose that we obtain the representative cycles $\{[z_p]\}_{p\in P_{ZZ}}$ for $\Int$. We \emph{cone} these cycles in the $P$-filtration $F$ of the
simplicial complex $K$ that induces $\M$. It means that, we use a dummy vertex $\omega$ and for every simplicial complex
$K_{\Fld_s(p)}$, $p\in P_{ZZ}$, we introduce the simplex $\sigma\cup {\omega}$ for
every $\sigma\in z_p$ if it is not already present. Let the new $P$-filtration be 
$\hat{F}$ which filters the coned complex denoted $\hat{K}$. The $P$-module
induced by $\hat{F}$ is the quotient module $\Q$. We unfold $\Q$ implicitly
by unfolding $\hat{F}$ to $\hat{F}_{ZZ}$ onto $P_{ZZ}$. 
Let $\Q_{ZZ}$ denote the zigzag module
obtained by this unfolding process. We compute a direct decomposition $\D$
of $\Q_{ZZ}$ from $\hat{F}_{ZZ}$ by the zigzag persistence algorithm in~\cite{DHM25}.
We say $\D$ is \emph{extendable} if there is a direct decomposition
$\D'$ of $\M_{ZZ}$ so that $\D'=\D\cup\{\Int\}$. 
\begin{proposition}
$\overline{\Int}$ is invertible if and only if $\D$ is extendable.
\label{prop:extendable}
\end{proposition}
\begin{proof}
    If $\D$ is extendable, we have $\M_{ZZ}=\Int\oplus \Q_{ZZ}$ by definition of
    extendability and the fact that $\D$ is a direct decomposition of $\Q_{ZZ}$.
    It follows that $\overline{\Int}$ is
    invertible becasue $\Q_{ZZ}$ is foldable by construction.

    For the opposite direction, assume that $\overline{\Int}$ is invertible.
    Since $\Int$ is foldable, $\Fld_s(\Int)$ exists and is a summand of
    $\M$ according to Theorem~\ref{thm:full-folding}(3). The quotient
    $\Q=\M/(\Fld_s(\Int))$ is a submodule of $\M$. This means the morphisms 
    $\Q(p\leq_P q)$ is a restriction of $\M(p\leq_P q)$ for every pairs $p\leq_P q$.
    Therefore, by the definition of unfolding, every morphism
    $\Q_{ZZ}(p\leq_{P_{ZZ}} q)$ is a restriction of $\M_{ZZ}(p\leq_{P_{ZZ}}q)$.
    We claim that $\D'=\D\cup{\Int}$ is a direct decomposition of $\M_{ZZ}$ establishing that $\D$ is extendable.

    To prove the claim, observe that, due to the quotienting, the vector generating $\Int(p)$ cannot lie in the vector space $\Q_{ZZ}(p)$ for every $p\in P_{ZZ}$. 
    Then, $\M_{ZZ}(p)=\Int(p)\oplus \Q_{ZZ}(p)$ for every $p\in P_{ZZ}$. Also, for
    every summand $\Int'$ of $\Q_{ZZ}$ in $\D$, the morphism $\Int'(p\leq_{P_{ZZ}} q)$ 
    is a restriction of $\M_{ZZ}(p\leq_{P_{ZZ}} q)$ because we already established
    that $\Q_{ZZ}(p\leq_{P_{ZZ}} q)$ is a restriction of $\M_{ZZ}(p\leq_{P_{ZZ}}q)$. Therefore, for every interval module $\Int_i\in \D\cup \{\Int\}$, every morphism $\Int_i(p\leq_{P_{ZZ}} q)$
    is a restriction of $\M_{ZZ}(p\leq_{P_{ZZ}} q)$ over all pairs $p\leq_{P_ZZ} q$
    while $\M_{ZZ}(p)=\oplus_i \Int_i(p)$. It follows that $\D'$ is a direct decomposition of $\M_{ZZ}$. 
   \end{proof}
We test if $\D$ is extendable by checking if the linear maps
connecting two vectors (cycle classes) $[z_p]$ and $[z_q]$ for all immediate
pairs $p,q\in P_{ZZ}$ in every interval module in $\D$ are preserved in
$\D'$. This means that if $p\leq_{P_{ZZ}} q$, the cycles $z_p$ and $z_q$ which
are homologous in the complex $\hat{K}_q$ remain
homologous in the complex $K_q$. This can be easily determined by checking the
annotations of the cycles $z_p$ and $z_q$ in $K_q$. We need to compute the
annotations of representative cycles for the interval modules in $\D$ 
in the complex $K_q$.
We have already seen that it costs $O(tn^{\omega})$ time in total. Thus, the extendability
of $\D$ can be checked in $O(tn^{\omega})$ time.

\cancel{
\begin{proposition}
    $\overline{\Int}$ is invertible if and only if the barcode of $\hat{F}_{ZZ}$ coincides with the barcode of $\bar{\Int}$.
\end{proposition}
\begin{proof}
   First, assume that $\bar\Int$ is invertible. By definition, this means that
   there is a summand, say $\U$, of $\M_{ZZ}$ so that $\U\cong \bar{\Int}$ and $\U$ is foldable.
   Since $\Int$ is foldable, By Theorem~\ref{thm:full-folding}(1), $\Fld_s(\U)$ exists and 
   $\M=\Fld_s(\U)\oplus\Fld_s(\int)$. Then $\Q\cong \Fld_s(\U)$, which implies that
   $\Q_{ZZ}\cong \Fld_s^{-1}(\Fld_s(\U))=\U$. Therefore, the barcode of $\Q_{ZZ}$ coincides with that of $\U$. The barcode of $\Q_{ZZ}$ is the barcode of $\hat{F}_{ZZ}$ and the barcode of $\U$ is same as that of $\bar{\Int}$.

   Now, assume that the barcode of $\hat{F}_{ZZ}$ coincides with that of $\bar{\Int}$.
   For contradiction, assume that $\bar{\Int}$ is not invertible. Since
   $\Fld_s(\Int)$ exists, we still have $\Q=\M/(\Fld_s\Int)$ well defined.
   Then, $\M_{ZZ}=\Int\oplus \Q_{ZZ}$. It means that $\Q_{ZZ}\cong \bar{\Int}$
   in which case $\bar\Int$ is invertible reaching a contradiction. 
\end{proof}
we have to check if $\overline{\Int}(p)=\overline{\Int}(p')$ for every point $p'\in\prt(p)$. The vector space $\overline{\Int}(p)$ for any point
$p\in P_{ZZ}$ is spanned by the basis $([z_p^1],\ldots,[z_p^r])$ where $z_p^i$, $i\in [r]$, are the $k$-cycles computed as representative cycles
for the interval modules $\Int_1,\ldots,\Int_r$ none of which equals $\Int$ and whose support contains $p$, that is, $\Int_i(p)\not = 0$. Let $L$ be the set of points defined in section~\ref{sec:convert}. For every point $p\in L$, we form an annotation matrix
$A_p$ whose columns represent the annotation of the basis elements $[z_p^1],\ldots,[z_p^r]$ for $\overline{\Int}(p)$. Then, for every point
$p'\in \prt(p)$, we consider the column vectors similarly formed by the
basis elements $[z_{p'}^1],\ldots, [z_{p'}^{r}]$ for $\overline{\Int}(p')$
and check if each of the vectors $[z_{p'}^i]$ is in the column span of $A_p$. If so, we have
$\overline{\Int}(p)=\overline{\Int}(p')$ for every $p'\in\prt(p)$ and otherwise not. For better time complexity, we augment $A_p$ with the column vectors made for $[z_{p'}^1],\ldots, [z_{p'}^{r}]$. Then, we reduce this augmented matrix of dimensions $O(g)\times O(gt)$ which takes at most $O(tg^\omega)$ time.}

\vspace{0.1in}
\noindent
{\bf Time complexity.}
Let the graph $G=(P,E(p))$ of the poset $P$ indexing the input filtration $F$ in {\sc GenRank}
have a total of $m$ vertices and edges. 
Recall that $K_p$ denote the complex at point $p\in P$. Let $e_{p,q}=|K_{p}\setminus K_{q}|$ where
$(p, q)$ is a directed edge in $G$, that is, $e_{p,q}$ is the number of simplices
that are added in the filtration going from $p$ to $q$. Let 
\begin{equation}
e=\sum_{(p, q)\in E(P)}e_{p,q} \mbox{ and } t=\max\{m,e\}.
\label{eq:tdef}
\end{equation}
The quantity $t$ is an upper bound on the size of the input
$P$-filtration and the size of the poset $P$.
Let every complex $K_p$ for every
point $p\in P$ has at most $n$ simplices. Then, step 1 of {\sc GenRank}
takes at most $O(t)$ time to unfold the poset to $P_{ZZ}$ with the procedure
described in section~\ref{sec:unfold}. Producing the zigzag filtration $F_{ZZ}$
takes at most $O(t)$ time. Step 2 takes at most $O(n^3)$ time to produce
the annotation matrix~\cite{DW22} for every complex $K_p$, $p\in P_{ZZ}$, giving a total
of $O(tn^3)$ time. Step 3 takes $O(t^2n)$ time to compute the zigzag barcode
with representatives by the algorithm in~\cite{DHM25}.

Step 4 for testing convertibility and converting of $\Int$ takes a total of $O(t^{\omega}n^2)$ time as discussed before. Testing invertibility of
$\bar\Int$ for every $\Int$ involves coning out the filtration $F$ and unfolding $\hat{F}$ which cannot
take more than $O(t)$ time. Computing the zigzag barcode 
of $\hat F_{ZZ}$ with representative cycles takes $O(t^2n)$ time by the
algorithm in~\cite{DHM25} giving a total time of
$O(t^2 n^2)$ for $O(n)$ full interval modules like $\Int$. The verification of the extendability
of the interval decomposition $\D$ of the quotient module $\Q_{ZZ}$ takes
$O(tn^{\omega})$ time, giving a total time of $O(tn^{\omega+1})$ over all.
Therefore, step 4 takes a total time $O(t^2n^2+tn^{\omega+1})$.
Accounting for all terms, we get: 
\begin{theorem}
Let $F$ be an input $P$-filtration of a complex with $n$ simplices where
$F$ and $P$ have size at most $t$. The algorithm {\sc GenRank} computes $\rk(\M)$ in 
$O(tn^{\omega+1}+t^{\omega}n^2)$ time
where $\M$ is the $P$-module induced by $F$. With $n=O(t)$, the bound becomes
$O(t^{\omega+2})$.
\label{thm:main}
\end{theorem}
\begin{proof}
The correctness of {\sc GenRank} follows from Theorem~\ref{thm:algo} because 
step 4 effectively either increases the count $\kappa(\D)$ or determines that $\kappa(\D)$ cannot
be increased in which case $\D$ is $s$-complete.
The time complexity
claim follows from our analysis.
\end{proof}

\section{Special case of degree-$d$ homology for $d$-complexes}
\label{sec:d-complex}
In this section, we show that when a $P$-module $\M$ is induced by applying the homology functor in degree $d\geq 0$ on a $P$-filtration of a $d$-dimensional simplicial complex, we have a much more efficient algorithm for computing $\rk(\M)$.
The key observation is that, in this case, the representatives for the interval modules in a direct decomposition of $\M_{ZZ}$ takes a special form. In particular, the following Proposition holds leading to Theorem~\ref{thm:d-complex}.
\begin{proposition}
    Let $\M$ be a $P$-module where $\M(p)=H_d(K_p)$ with $K_p$ being a simplicial $d$-complex. For an interval module $\Int^{[p_b,p_d]}$ in a direct decomposition of $\M_{ZZ}$ and for any two points $p_i,p_j\in [p_b,p_d]$, the
    representative $d$-cycles $z_{p_i}$ and $z_{p_j}$ are the same.
    \label{prop:d-module}
\end{proposition}
\begin{proof}
     First assume that $p_i$ and $p_j$ are immediate points in $P_{ZZ}$ and without loss of generality let $p_i\rightarrow p_j$. According to the definition of representatives, we must have the homology classes $[z_{p_i}]$ and $[z_{p_j}]$ homologous in $K_{p_j}$. Since $K_{p_j}$ is a $d$-complex, $[z_{p_i}]=[z_{p_j}]$ only if $z_{p_i}=z_{p_j}$ as chains. It follows by transitivity that this is true even if $p_i$ and $p_j$ are not immediate.
\end{proof}
It follows from the above proposition that a full interval module $\Int$ in a direct decomposition of $\M_{ZZ}$ is foldable.
Next proposition says that even the complement $\overline\Int$ is invertible. Then, applying Theorem~\ref{thm:full-folding}(3), we can claim that $\Fld_s(\Int)$ is a full
interval summand of $\M$.
\begin{proposition}
    Let $\Int$ be any full interval module in a direct decomposition of $\M_{ZZ}$ where $\M_{ZZ}$ is constructed as in Proposition~\ref{prop:d-module}. Then, $\Int$ is foldable and $\overline\Int$ is invertible.
    \label{prop:d-full}
\end{proposition}
\begin{proof}
    $\Int$ is foldable due to Proposition~\ref{prop:d-module}. We need to show that $\overline{\Int}$ is invertible.

    For the full module $\Int$, there is a single $d$-cycle that forms a fixed representative at each point in $P_{ZZ}$ (Proposition~\ref{prop:d-module}). Let $z$ be such a $d$-cycle and $\sigma$ be any $d$-simplex in $z$. Delete $\sigma$ from the complex $K_p$ for every $p\in P_{ZZ}$. Let $\M_{ZZ}^-$ denote the $P_{ZZ}$-module induced by the homology functor in degree $d$ on the zigzag filtration $F_{ZZ}^-$ obtained from the original filtration $F_{ZZ}$ by deleting the simplex $\sigma$ everywhere. It is easy to verify that $\M_{ZZ}^-$ is $s$-foldable because $K_p\setminus \sigma=K_{p'}\setminus \sigma$ for every $p,p'\in P_{ZZ}$ with $s(p)=s(p')$. 
    A direct decomposition of $\M_{ZZ}$ is obtained from a direct decomposition of $\M_{ZZ}^-$ by adding a full bar with the representative $z$ or equivalently the summand $\Int$. It follows that $\overline{\Int}$ is invertible.
\end{proof}

\begin{theorem}
    Let $\M$ be a module constructed as in Proposition~\ref{prop:d-module} from a $P$-filtration $F$ of a $d$-complex where $P$ and $F$ have size at most $t$. Then, $\rk(\M)$ is the number of full interval modules in any direct decomposition of $\M_{ZZ}$ which can be computed in $O(t^\omega)$ time. 
    \label{thm:d-complex}
\end{theorem}
\begin{proof}
    By Proposition~\ref{prop:d-full} and Theorem~\ref{thm:full-folding}(3), we get that $\kappa(\D)=\tau(\D)$ for any direct
    decomposition $\D$ of $\M_{ZZ}$. Then, it follows from
    Proposition~\ref{prop:convert} that $\rk(\M)=\kappa(\D)$, which is exactly equal to the number of full interval modules in any direct decomposition of $\M_{ZZ}$. Therefore, $\rk(\M)$ can be simply obtained by computing the zigzag barcode of $\M_{ZZ}$ (no need of computing the representatives). This can be done in $O(t^\omega)$ time with the fast zigzag algorithm~\cite{DH22}. 
\end{proof}
If $\M$ is induced by a $P$-filtration of a graph, then $\rk(\M)$ can be computed even faster. 
Observe that every $1$-cycle that represents a full bar must continue to be present from the initial graph $G_{p_0}$ at $p_0\in P$ to the final graph $G_{p_m}$ at $p_m\in P$. This suggests the following algorithm. Take $G_{p_m}$; delete all edges and vertices that are deleted as one moves along $P_{ZZ}$ from $p_0$ to $p_m$, but do not insert any of the added edges and vertices. In the final graph $G_{p_m}'$ thus obtained (which may be different from $G_{p_m}$ because we ignore the inserted edges and vertices along $P_{ZZ}$), we compute the number of independent $1$-cycles. This number can be computed by a depth first search in $G_{p_m}'$ in linear time. Then, as an immediate corollary we have Theorem~\ref{thm:graph}.
\begin{theorem}
    If $\M$ is induced by degree-$1$ homology of a $P$-filtration $F$ of a graph with a total of $n$ vertices and edges, then $\rk(\M)$ can be computed in $O(n+t)$ time where $P$ and $F$ have size at most $t$.
    \label{thm:graph}
\end{theorem}
\section{Conclusions and discussions}
\label{sec:conclusion}
\cancel{One parameter slicing introduced to define matching
distances between multiparameter persistence modules~\cite{BL21,Cerri,KLO20} loses
information. The zigzag straightening proposed here is lossless
because it retains all information of the original module.
Can it be used to define more discriminating and stable
invariants--a question we pose for future investigation. 
}
Analyzing a mutliparameter persistence module with the help of
one parameter persistence modules is not new. It has been introduced in the
context of computing 
matching distances~\cite{BL21,Cerri,KLO20} for $2$-parameter
 modules.
The unfolding/folding technique proposed here offers a different
slicing technique. By producing one zigzag path instead of multiple slices,
the unfolding preserves the structural maps of the original module $\M$ 
in a lossless manner. 
We showed how to use them for reconstructing full interval modules and hence
for computing the generalized rank. It will be interesting to see what other 
invariants one can reconstruct using the folding/unfolding of persistenec modules. A natural
candidate would be to compute the limit and colimit of the 
original module from its zigzag straightening. Recent advances in
zigzag persistence computations~\cite{DH22,maria2014zigzag,milosavljevic2011zigzag} can then
be taken advantage of for computing limits and colimits.

A natural question to ask is that if the approach laid out in this paper works
even when the original input module is unfolded into another module that is
not necessarily a zigzag module. This could lead to a more generalized theory
for unfolding. Unfortunately, this question does not have a positive answer. The
main bottleneck arises while attempting to extend Proposition~\ref{prop:complete-interval}. This
proposition says that any full interval module that is a submodule of a zigzag module
$\M_{ZZ}$ is a sum of the limit modules in a direct decomposiiton of
$\M_{ZZ}$. We cannot extend this result to every module $\M$ that is not a zigzag module. The proof of Theorem~\ref{thm:algo} uses Proposition~\ref{prop:convert}
which depends on this result in a fundamental way. To see the obstacle,
consider the module $\M$ in Figure~\ref{fig:unfoldingEx3} (left,top). This module
has a full interval module $\Int$ as a submodule, namely $\Int$ is given by
the vector spaces spanned by the vector $v$ at every point and the internal map as
the identity. However, the module $\M$ itself is not decomposable and hence
has no limit module as a summand. Thus, the full module $\Int$ has no expression
as a sum of limit modules that are summands of $\M$. The situation does not
improve even if $\Int$ is a summand instead of being merely a submodule. The
module $\M$ in Figure~\ref{fig:unfoldingEx3} (left,bottom) has a full interval
module $\Int$ as a summand whose vector spaces are spanned by the vector $u+v$ at
every point. Also, a different decomposition of $\M$ can have a full interval module
$\Int'$ as a summand whose vector spaces are spanned by the vector
$u$ at every point. But, $\Int$ cannot be expressed as a sum of limit modules
in this decomposition which has $\Int'$ as the only limit module.

The special structure
of the zigzag modules allows us to express any full submodule 
as a sum of limit modules that are summands, which in turn lets our algorithm to
check if a full module in the current decomposition is convertible or not.

\subsection*{Acknowledgment.} We thank Iason Papadopoulos of University of Bremen for useful comments on the paper.
\bibliography{biblio}

\cancel{
\begin{remark}\label{rem:canonical projection}Let $\Pb$ be a finite and connected poset. Let $Q$ be a finite and connected subposet of $\Pb$. Let us fix any $F:\Pb \rightarrow \vect$.
\begin{enumerate}[label=(\roman*)]
    \item  For any cone $\left(L', (\pi_p')_{p\in \Pb}\right)$ over $F$, its restriction $\left(L', (\pi_p')_{p\in Q}\right)$ is a cone over the restriction $F\vert_Q:Q\rightarrow \vect$. Therefore, by the terminal property of the limit  $\left(\varprojlim F\vert_{Q}, (\pi_q)_{q\in Q}\right)$, there exists the unique morphism $u:L' \rightarrow \varprojlim F\vert_{Q}$ such that  $\pi_q'=\pi_q\circ u$ for all $q\in Q$.\label{item:canonical projection 1}
    \item  For any cocone $\left(C', (i_p')_{p\in \Pb}\right)$ over $F$, its restriction $\left(C', (i_p')_{p\in Q}\right)$ is a cocone over the restriction $F\vert_Q:Q\rightarrow \vect$. Therefore, by the initial property of $\varinjlim F\vert_{Q}$, there exists the unique morphism $u:\varinjlim F\vert_{Q} \rightarrow C'$ such that $i'_q=u\circ i_q$ for all $q\in Q$.\label{item:canonical projection 2}
     \item By the previous two items, there exist linear maps $\pi:\varprojlim F\rightarrow \varprojlim F\vert_{Q}$ and $\iota:\varinjlim F\vert_{Q}\rightarrow \varinjlim F$ such that
     $\psi_F=\iota \circ \psi_{F\vert_{Q}} \circ \pi.$\label{item:canonical projection 3}
\end{enumerate}
Therefore, $\rank(F)=\rank(\psi_{F})\leq \rank (\psi_{F\vert_{Q}})=\rank(F\vert_{Q})$.
\end{remark}
}
\cancel{
\section{Missing proof in section~\ref{sec:zigzag}}

\begin{proof}[Proof of Proposition~\ref{obs:limit-module}]
    Let $\Int_1,\ldots,\Int_k$ be the set of interval modules in the considered
    direct decomposition of the $P_{ZZ}$-module $\M_{ZZ}$. For any $p\in P_{ZZ}$, since $\M_{ZZ}(p)=\text{span}(\bs_{\Int_1(p)},\cdots,\bs_{\Int_k(p)})$, there exist uniquely determined
    $\alpha^p_i\in\mathbb{F}$ so that
    $\bs_{\Int(p)}=\sum_i \alpha^p_i \bs_{\Int_i(p)}$. Let $L_p=\{\Int_i\,|\, \alpha^p_i\not=0\}$
    and $L'=\cup_p L_p$. It is not difficult to show that
    $\alpha^p_i=\alpha^q_i:=\alpha_i$ for any two points $p,q$
    in the support of $\Int_i$.
    Then, we can write $L'=\{\Int_i \,|\, \alpha_i\not = 0\}$. We claim that
    $L'\subseteq L$. If not, there is an interval module $\Int'\in L'$ that is not
    a limit module, that is, $\Int'$ has an end point $p_j\not \in \{p_0, p_m\}$ satisfying either of the following two cases: (i) the arrow $f_{p_{j-1}}$ is forward and the cokernel of $f_{p_{j-1}}$ restricted to $\Int'(p_{j-1})$ is non-zero.
    It follows that the cokernel of $f_{p_{j-1}}$ restricted to $\Int(p_{j-1})$ is also non-zero. This
    is impossible as $\Int$ is a full interval module and $p_j\not\in\{p_0,p_m\}$, (ii) the arrow $f_{p_{j}}$ is backward: again, we reach an impossibility with a similar argument. So, $L'\subseteq L$ and it follows that $\Int=\sum_{\Int_i\in L'}\alpha_i\Int_i$ establishing the claim of the proposition. 
\end{proof}
}

\cancel{
\begin{proof}[Proof of Theorem~\ref{thm:full-folding}]
The modules $\Int$ and $\overline\Int$ satisfy the conditions in Proposition~\ref{prop:full-folding}. Thus, $\Fld_s (\Int)$ is a summand of $\M$.
Since $\Int$ is full, the module $\Fld_s(\Int)$
has full support on $P$ with its internal morphisms being non-zero  between  one dimensional vector spaces.
The claim follows.
\end{proof}
}
\cancel{
\begin{proof}[Proof of Proposition~\ref{prop:complete-rank}]
Let $\M=\Int_1\oplus\cdots\oplus\Int_r\oplus \M'$ where $\Int_1,\ldots,\Int_r$
are full interval modules. By Theorem~\ref{thm:rk}, $r=\rk(\M)$. Then, $\M_{ZZ}=\Fld_s^{-1}(\Int_1)\oplus\cdots\oplus\Fld_s^{-1}(\Int_r)\oplus \Fld_s^{-1}(\M')$ by Theorem~\ref{thm:full-folding}. Furthermore, each of $\Fld_s^{-1}(\Int_i)$, $1\leq i\leq r$, is a full module because each $\Int_i$ is so. By definition both $\Fld_s^{-1}(\Int_i)$ and its complement are foldable. Therefore, each $\Fld_s^{-1}(\Int_i)$ is $s$-complete. We observe that $\Fld_s^{-1}(\M')$ cannot have a summand which is $s$-complete because if that were the case, $\M_{ZZ}$ would have more than $r$ $s$-complete interval modules as its summand each of which would fold to a full interval summand in $\M$ (Theorem~\ref{thm:full-folding}). This is not possible because in that case $\M$ would have more than $r$ summands that are full intervals, an impossibility according to Theorem~\ref{thm:rk}.
It follows that $\M_{ZZ}=\Fld_s^{-1}(\Int_1)\oplus\cdots\oplus\Fld_s^{-1}(\Int_r)\oplus \Fld_s^{-1}(\M')$ is an $s$-complete decomposition of $\M_{ZZ}$. 
\end{proof}
}

\cancel{
\begin{proof}[Proof of Theorem~\ref{thm:main}]
The correctness of {\sc GenRank} follows from Theorem~\ref{prop:convert} because it checks the convertibility of the full interval modules and foldability of their complements as stated in this theorem. The time complexity
claim follows from our analysis.
\end{proof}

\section{Missing details in section~\ref{sec:d-complex}}
\label{appendix:d-complex}
\begin{proof}[Proof of Proposition~\ref{prop:d-module}]
     First assume that $p_i$ and $p_j$ are immediate points in $P_{ZZ}$ and without loss of generality let $p_i\rightarrow p_j$. According to the definition of representatives, we must have the homology classes $[z_{p_i}]$ and $[z_{p_j}]$ homologous in $K_{p_j}$. Since $K_{p_j}$ is a $d$-complex, $[z_{p_i}]=[z_{p_j}]$ only if $z_{p_i}=z_{p_j}$ as chains. It follows by transitivity that this is true even if $p_i$ and $p_j$ are not immediate.
\end{proof}
It follows from the above proposition that a full interval module in a direct decomposition of $\M_{ZZ}$ is foldable.
Next proposition says that even the complement $\overline\Int$ is foldable. Then, applying Theorem~\ref{thm:full-folding}(1), we can claim that $\rk(\M)$ is equal to the number of full interval modules in $\M_{ZZ}$.

\begin{proposition}
    Let $\Int$ be a full interval module in a direct decomposition of $\M_{ZZ}$ where $\M_{ZZ}$ is constructed as in Proposition~\ref{prop:d-module}. Then, both $\Int$ and $\overline\Int$ are foldable.
    \label{prop:d-full}
\end{proposition}
\begin{proof}
    $\Int$ is foldable due to Proposition~\ref{prop:d-module}. Let $p,p'\in P_{ZZ}$ be any two points with $s(p)=s(p')$. We need to show that $\overline{\Int}(p)=\overline{\Int}(p')$.

    For the full module $\Int$, there is a single $d$-cycle that forms the representative at each point in $P_{ZZ}$ (Proposition~\ref{prop:d-module}). Let $z$ be such a $d$-cycle and $\sigma$ be any $d$-simplex in $z$. Delete $\sigma$ from the complex $K_p$ for every $p\in P_{ZZ}$. Let $\M_{ZZ}^-$ denote the $P_{ZZ}$-module induced by the homology functor in degree $d$ on the zigzag filtration $F_{ZZ}^-$ obtained from the original filtration $F_{ZZ}$ by deleting the simplex $\sigma$ everywhere. It is easy to verify that $\M_{ZZ}^-$ is $s$-foldable because $K_p\setminus \sigma=K_{p'}\setminus \sigma$ for every $p,p'\in P_{ZZ}$ with $s(p)=s(p')$. 
    A direct decomposition of $\M_{ZZ}$ is obtained from a direct decomposition of $\M_{ZZ}^-$ by adding a full bar with the representative $z$. Therefore, $\overline{\Int}$ is equal to $\M_{ZZ}^-$ and thus foldable.
\end{proof}
\begin{proof}[Proof of Proposition~\ref{thm:d-complex}]
    By Proposition~\ref{prop:d-full} and Theorem~\ref{thm:full-folding}(1), $\rk(\M)$ is exactly equal to the number of full interval modules in any direct decomposition of $\M_{ZZ}$. Therefore, $\rk(\M)$ can be simply obtained by computing the zigzag barcode of $\M_{ZZ}$ (no need of computing the representatives). This can be done in $O(t^\omega)$ time with the fast zigzag algorithm~\cite{DH22}. 
\end{proof}

If $\M$ is induced by a $P$-filtration of a graph, then $\rk(\M)$ can be computed even faster. 
Observe that every $1$-cycle that represents a full bar must continue to be present from the initial graph $G_{p_0}$ at $p_0\in P$ to the final graph $G_{p_m}$ at $p_m\in P$. This suggests the following algorithm. Take $G_{p_m}$; delete all edges and vertices that are deleted as one moves along $P_{ZZ}$ from $p_0$ to $p_m$, but dont insert any of the added edges and vertices. In the final graph $G_{p_m}'$ thus obtained (which may be different from $G_{p_m}$ because we ignore the inserted edges and vertices along $P_{ZZ}$), we compute the number of independent $1$-cycles. This number can be computed by a depth first search in $G_{p_m}'$ in linear time. Then, as an immediate corollary we have Theorem~\ref{thm:graph}.
}

\end{document}